\documentclass[12pt, leqno]{article}
\usepackage{amsmath}
\usepackage{amsthm}
\usepackage{amssymb}
\usepackage{latexsym}
\usepackage{graphicx}
\allowdisplaybreaks[1]
\usepackage{amsfonts}
\usepackage{ascmac}
\usepackage{color}
\usepackage{enumerate}

\theoremstyle{definition}
\theoremstyle{plain}
\allowdisplaybreaks[1]
\date{}
\newtheorem{Thm}{Theorem}[section]
\newtheorem{Prop}[Thm]{Proposition}
\newtheorem{Lemma}[Thm]{Lemma}

\newcommand{\dis}{\displaystyle}

\newcommand{\norm}{\parallel}

\newcommand{\Z}{{\mathbb Z}}

\newcommand{\N}{{\mathbb N}}
\newcommand{\R}{{\mathbb R}}

\newcommand{\ep}{\varepsilon }
\newcommand{\2}{\frac{1}{2} }
\newcommand{\wto}{\rightharpoonup}

\newcommand{\nnorm}{|\!|\!|}

\newcommand{\D}{\mathcal{D}}

\def\text#1{\mbox{#1 }}

\title{\bf On convergence of Chorin's projection method to a Leray-Hopf weak solution}
\author{Hidesato Kuroki
\footnote{A part of this work was done, when he belonged to Department of Mathematics, Faculty of Science and Technology, Keio University, 3-14-1 Hiyoshi, Kohoku-ku, Yokohama, 223-8522, Japan. }
\,\, and \,\,  Kohei Soga
\footnote{Department of Mathematics, Faculty of Science and Technology, Keio University, 3-14-1 Hiyoshi, Kohoku-ku, Yokohama, 223-8522, Japan. E-mail:  soga@math.keio.ac.jp 
}}
\begin{document}
\maketitle
\begin{abstract} 
\noindent The projection method to solve the incompressible Navier-Stokes equations was first studied by Chorin [Math. Comp., 1969] in the framework of a finite difference method and Temam [Arch. Rational Mech. and Anal., 1969] in the framework of a finite element method. Chorin showed convergence of approximation and its error estimates in problems with the periodic boundary condition  assuming existence of a $C^5$-solution, while Temam demonstrated an abstract argument to obtain a Leray-Hopf weak solution in problems on a bounded domain with the no-slip boundary condition. In the present paper, the authors extend Chorin's result  with full details to obtain convergent finite difference approximation of a Leray-Hopf weak solution  to the incompressible Navier-Stokes equations  on an arbitrary bounded Lipschitz domain of $\R^3$  with the no-slip boundary condition and an external force. We prove  unconditional solvability of our implicit scheme and strong $L^2$-convergence (up to subsequence) under the scaling condition $ h^{3-\alpha}\le \tau$ (no upper bound is necessary), where $h,\tau$ are space, time discretization parameters, respectively, and $\alpha\in(0,2]$ is any fixed constant. The results contain  a  compactness method based on  a new interpolation inequality for step functions.           
 
\medskip

\noindent{\bf Keywords:}  incompressible Navier-Stokes equations; Leray-Hopf weak solution; projection method;  finite difference scheme  \medskip

\noindent{\bf AMS subject classifications:}   35Q30; 35D30; 65M06
\end{abstract}
%
\setcounter{section}{0}
\setcounter{equation}{0}
\section{Introduction}
We consider the incompressible Navier-Stokes equations on a bounded domain of $\R^3$
\begin{eqnarray}\label{NS}
 \left \{
\begin{array}{lll}
\,\,\,\,\, v_t&=& -(v\cdot \nabla)v +\Delta v+f -\nabla p\mbox{\quad in $(0,T]\times\Omega $,}
\medskip\\
\nabla\cdot v &=&0\mbox{\quad\quad\quad\quad\quad\quad\quad\quad\quad\quad\quad\,\,\, in $(0,T]\times\Omega$, \qquad\qquad\quad}
\medskip\\
v(0,\cdot)&=&v^0\mbox{\qquad\qquad\qquad\qquad\qquad\,\,\,\,\,\,\, in $\Omega$},
\medskip\\
\,\,\,\,\,v&=&0\mbox{\qquad\qquad\qquad\qquad\qquad\,\,\,\,\,\,\,\,\, on $\partial \Omega$},
\end{array}
\right.\\\nonumber  
\,\,\,\,\,\,\,\, \Omega\subset \R^3 \mbox{ is a bounded connected open set with a Lipschitz boundary,}
\end{eqnarray}
where $v=v(t,x)$ is the velocity, $p=p(t,x)$ is the pressure, $f=f(t,x)$ is a given external force, $T$ is an arbitrary positive number, $v^0$ is initial data and $v_t=\partial_t v$, $v_{x_j}=\partial_{x_j}v$, etc., stand for  the partial (weak) derivatives of $v(t,x)$. 
Let $f$ and $v^0$ be arbitrarily taken as  
$$\mbox{$f\in L^2_{loc}([0,\infty);L^2(\Omega)^3)$,\quad $v^0\in L^2_{\sigma}(\Omega)$}.$$ 
Here, $C^r_{0}(\Omega)$  is the family of $C^r$-functions\,:\,$\Omega\to\R$ with a compact support; $C^r_{0,\sigma}(\Omega):=\{v\in C^r_0(\Omega)^3\,|\,\nabla\cdot v=0\}$;  $H^1_0(\Omega)$ is the closure of $C^\infty_0(\Omega)$ with respect to the norm  $\norm \cdot\norm_{H^1(\Omega)}$; $L^2_{\sigma}(\Omega)$ (resp. $H^1_{0,\sigma}(\Omega)$) is the closure of $C^\infty_{0,\sigma}(\Omega)$ with respect to the norm $\norm\cdot\norm_{L^2(\Omega)^3}$ (resp. $\norm \cdot\norm_{H^1(\Omega)^3}$). 

A function $v=(v_1,v_2,v_3):[0,T]\times\Omega\to\R^3$ is called a {\it Leray-Hopf weak solution} of \eqref{NS}, if 
\begin{eqnarray}\nonumber
&&v\in L^\infty([0,T];L^2_{\sigma}(\Omega))\cap L^2([0,T];H^1_{0,\sigma}(\Omega)),\\\label{weak-form-NS}
&&-\int_\Omega v^0(x)\cdot\phi(0,x)dx- \int_0^T\int_\Omega v(t,x)\cdot \partial_t\phi(x,t)dxdt \\\nonumber
&&=-\sum_{j=1}^3\int_0^T\int_\Omega v_j(t,x)\partial_{x_j}v(t,x)\cdot\phi(t,x)dxdt\\\nonumber
&&\quad -\sum_{j=1}^3\int_0^T\int_\Omega\partial_{x_j}v(t,x)\cdot\partial_{x_j}\phi(t,x)dxdt\\\nonumber
&&\quad +\int_0^T\int_\Omega f(t,x)\cdot\phi(t,x)dxdt\quad \mbox{ for all $\phi\in C^\infty_0([-1,T);C^\infty_{0,\sigma}(\Omega))$,}
\end{eqnarray} 
where $x\cdot y:=\sum_{i=1}^3x_iy_i$ for $x,y\in\R^3$. Note that we have
$$-\sum_{j=1}^3\int_0^T\int_\Omega v_j(t,x)\partial_{x_j}v(t,x)\cdot\phi(t,x)dxdt= \sum_{j=1}^3 \int_0^T\int_\Omega    v_{ j}(t,x)v(t,x) \cdot  \partial_{x_j}\phi(t,x)    dxdt $$
for $v\in L^\infty([0,T];L^2_{\sigma}(\Omega))\cap L^2([0,T];H^1_{0,\sigma}(\Omega))$.
Since the notion of Leray-Hopf weak solutions was introduced, a vast amount of research and achievement has been made to understand properties of the solutions (see, e.g., \cite{Sohr} and \cite{Tsai} with references therein). In the huge literature on the Navier-Stokes equations, let us re-discuss how to prove the existence of a Leray-Hopf solution of \eqref{NS}.

The existence of a Leray-Hopf weak solution of (\ref{NS}) was first proved by Hopf \cite{Hopf} through the Galerkin approximation. After that, Ladyzhenskaya  \cite{Ladyzhenskaya} developed fully-discrete finite difference approximation (discrete in both time and space)  of a Leray-Hopf weak solution. She proposed several discretization schemes and their a priori estimates. Her idea is to directly discretize (\ref{NS}) with an implicit formulation including $p$ and the divergence-free constraint. To prove convergence of the approximation, it is  essential to verify not only weak convergence of a sequence of approximate solutions but also its strong convergence. It turns out that a proof of the strong convergence is a non-trivial delicate issue (Ladyzhenskaya shortly announces ideas in \cite{Ladyzhenskaya}, but there is no proof). Chorin \cite{Chorin} developed Ladyzhenskaya's idea by separating the divergence-free constraint from \eqref{NS}. His idea, which is called {\it Chorin's projection method}, is to introduce a discrete version of the Helmholtz-Hodge decomposition and to formulate a finite difference version of the projected Navier-Stokes equations 
\begin{eqnarray}\label{projected}
v_t=\mathcal{P}( -(v\cdot \nabla)v+\Delta v+f),     
\end{eqnarray}
where $\mathcal{P}$ is the Helmholtz-Hodge decomposition operator. We remark that Ladyzhenskaya's scheme and Chorin's are  not equivalent because of the implicit formulation of schemes and nonlinearity of  (\ref{NS}). Chorin showed convergence and error estimates of his scheme applied to problems  on a $3$-dimensional torus, assuming that there exists an exact solution in the $C^5$-class. The main ingredient of his convergence proof is the Taylor expansion of the exact $C^5$-solution. In \cite{Temam-1} and \cite{Temam-2}, Temam developed  fully-discrete approximation of (\ref{NS}) in a rather abstract framework of a finite element method, which yields a Leray-Hopf weak solution (he dealt with not only a version of Chorin's projection method but also an artificial compressibility method). He introduced a nice trilinear form to handle approximation of the nonlinear term in (\ref{NS}) to obtain suitable a priori $L^2$-estimates. Then, in order to prove strong convergence,  he exploited a compactness theorem (see Section 2 of Chapter 3  in \cite{Temam-book}), which requires  a sequence $\{v_m\}_{m\in\N}$ of approximate solutions to satisfy the estimate: for some constant $\gamma>0$ and $A>0$,   
\begin{eqnarray}\label{1-compact}
\int_\R|s|^{2\gamma} \norm(\mathcal{F}_t v_m)(s)\norm_{L^2(\Omega)^3}^2 ds \le A \mbox{\quad  for all $m\in\N$},
\end{eqnarray}
where $\mathcal{F}_t v_m$ stands for the Fourier transform with respect to $t$ of $v_m(t,x)$ extended by $0$ outside $[0,T]$ (this is an estimate of the fractional time-derivative). 

Chorin's or Temam's projection method has been further developed and frequently employed in numerical analysis of the incompressible Navier-Stokes equations: We refer to \cite{Rannacher} for the fact that Chorin's projection method can be interpreted as a pressure stabilization method; Chapter 3 in \cite{Temam-book} for semi-discrete  approximation (discrete in time and continuous in space) that is also an effective approach to construct a Leray-Hopf weak solution;   \cite{Shen} for error estimates in semi-discrete  approximation in the class of strong solutions; \cite{Guermond} for an abstract functional analytic treatment of finite element projection methods with error estimates  in the class of strong solutions.

In this paper, we show every detail of a fully-discrete finite difference method along Chorin's idea, obtaining a new elementary proof of the existence of a Leray-Hopf weak solution to \eqref{NS}. Unlike Temam's framework, we stick to finite difference equations directly derived from \eqref{NS} and \eqref{projected}. Our proof tells how  to solve the finite difference equations only by four basic arithmetic operations, where one can see in a very elementary way how the incompressible Navier-Stokes equations evolve as a Leray-Hopf weak solution.  
We prove that {\it our scheme (to be implicit) is unconditionally solvable and stable, and that the scheme is strongly convergent (up to subsequence) as the space, time discretization parameters $h,\tau$ tend to $0$ under the scaling condition  $h^{3-\alpha}\le\tau$ (no upper bound is necessary), where  $\alpha\in(0,2]$ is any fixed constant.} Following Chorin's construction of the discrete Helmholtz-Hodge decomposition on a torus, we obtain a discrete Helmholtz-Hodge decomposition operator $P_h$ on the grid $\Omega_h$ of $\Omega$ with the $0$-boundary condition both for the divergence free part and potential part. Our argument proceeds as 
\begin{enumerate}
\item The intermediate velocity $u^{n+\2}:\Omega_h\to\R^3$ ($n=0,1,\ldots$ is the time index) is obtained by the discrete Navier-Stokes equations on $\Omega_h$ with the boundary condition $u^{n+\2}=0$ on $\partial \Omega_h$,
\item The end-of-step velocity $u^{n+1}:\Omega_h\to\R^3$ is defined  as $u^{n+1}:=P_hu^{n+\2}$ with a discrete Helmholtz-Hodge decomposition operator $P_h$.
\item The step functions $u_\delta,v_\delta:[0,T]\times\Omega\to\R^3$,  $\delta=(h,\tau)$ obtained by $u^n, u^{n+\2}$, respectively,  converge weakly to a common  function $v\in L^2([0,T];H^1_{0,\sigma}(\Omega))$.
\item $v_\delta$ (NOT $u_\delta$) converges strongly to $v$, and $v$ is a Leray-Hopf solution.   
\end{enumerate}    
In order to prove strong convergence, we introduce  a simple compactness method that works essentially with estimates for weak convergence (no additional estimate like  \eqref{1-compact} is necessary). 
 Our method (Lemma 6.1 and Theorem 6.2 in  Section 6) is based on an interpolation inequality for a sequence of step functions, which is seen as a discrete version of  the following well-known compactness   (see, e.g.,  Lemma 2.1 and Theorem 2.1 of Chapter 3 in \cite{Temam-book} for more abstract statements known as the Aubin-Lions lemma;  see also \cite{GL} and \cite{Gallouet1} for discrete analogues of the Aubin-Lions lemma): 
\begin{Prop}\label{1111}
\begin{enumerate}
\item For each $\eta>0$, there exists a constant $c_\eta$ such that 
\begin{eqnarray}\label{star1}
\norm u\norm_{L^2(\Omega)^3}\le \eta \norm u\norm_{H^1(\Omega)^3}+c_\eta\norm u\norm_{H^1(\Omega)^3{}^\ast}\quad\mbox{for all $u\in H^1_{0,\sigma}(\Omega)$},
\end{eqnarray}
where $H^1(\Omega)^3{}^\ast$ is the dual space of $H^1(\Omega)^3$.
\item   Suppose that $\{v_m\}_{m\in\N}$ is a bounded sequence of $L^\alpha([0,T];H^1_{0,\sigma}(\Omega))$ and  $\{\partial_tv_m\}_{m\in\N}$ is a bounded sequence of $L^\beta([0,T];H^1_{0,\sigma}(\Omega)^\ast)$ for some $1<\alpha,\beta<\infty$. Then, $\{v_m\}_{m\in\N}$ contains a subsequence that is convergent in $L^\alpha([0,T];L^2(\Omega)^3)$.
\end{enumerate}
\end{Prop}
\noindent In compactness arguments for discrete problems,  the essential points are how to choose discrete analogue of $H^1_{0,\sigma}(\Omega)$ and how to estimate the discrete time derivative, which depends highly on schemes. In our case, difficulty is that functions of the discrete divergence free constraint do not belong to $H^1_{0,\sigma}(\Omega)$ and the constraint varies along with the mesh size; Moreover, we face the absence of discrete $L^2([0,T];H^1(\Omega)^3)$-bound  of $u^n$ (this is because $P_h$ and discrete differentiation are not commutative). Nevertheless, we manage to prove strong convergence of the intermediate velocity $u^{n+\2}$ by means of the facts that $u^{n+\2}$ has discrete  $L^2([0,T];H^1(\Omega)^3)$-bound and that   $u^{n+\2}$ is asymptotically divergence free as $h,\tau\to0$ even though it is not discrete divergence free.   We show a sequence-wise discrete version of \eqref{star1} (i.e., we look for $c_\eta$ for each given sequence).  Thanks to the highly specialized form of our interpolation inequality, the discrete time derivative can be easily estimated through  our difference equations and estimates for weak convergence. It seems that compactness argument with our interpolation inequality is widely applicable to a proof of convergence  of fully discrete numerical schemes.

We assume that the boundary of $\Omega$ is Lipschitz, because we need to use the following fact (see, e.g., Theorem 1.6 and Remark 1.7  of Chapter 1 in \cite{Temam-book}):   
\begin{Prop}\label{1key-lemma}
Let $\Omega\subset\R^3$ be a bounded open set with a Lipschitz boundary $\partial\Omega$. Then, $H^1_{0,\sigma}(\Omega)$ coincides with $\tilde{H}^1_{0,\sigma}(\Omega):=\{v\in H^1_0(\Omega)^3\,|\,\nabla\cdot v=0\}$.       
\end{Prop}
\noindent Unlike Galerkin approximation, our method constructs a Leray-Hopf solution by a limit of sequences of  step functions compactly  supported in $\Omega$ with a discrete  divergence-free constraint. Hence,  the direct  consequence about regularity is that the limit function belongs to $L^2([0,T];\tilde{H}^1_{0,\sigma}(\Omega))\cap L^\infty([0,T];L^2(\Omega)^3)$. If $\Omega$ is a bounded Lipschitz domain, Proposition \ref{1key-lemma} yields the necessary regularity for the limit to be a Leray-Hopf solution. 

Chorin uses the central difference to define the  discrete gradient and divergence at each grid point, which is convenient to obtain higher accuracy through the Taylor expansion. We use  the forward difference for the gradient and the backward difference for the divergence (see Section 2), which simplifies the whole argument (e.g., this allows us to have  the $0$-boundary condition to the potential part instead of the $0$-mean condition in our  discrete Helmholtz-Hodge decomposition). Minor modification of our argument yields similar results with the central difference. 


We briefly discuss outlook of our result and related recent works. First of all we remark that the purpose of mathematical analysis of numerical methods is not only to provide actual computational techniques in practical situations but also to provide mathematical tools to establish rigorous results such as existence and uniqueness of solutions for (highly) nonlinear problems. We summarize the following three points as outlook of numerical analysis of the Navier-Stokes equations:        
\begin{itemize}
\item[(i)] In principle, since uniqueness of a Leray-Hopf weak solution is an open problem, a new existence proof might imply new knowledge on Leray-Hopf weak solutions, namely a new method might or might  not capture a Leray-Hopf weak solution which is different from the ones obtained by Hopf,  Temam or Caffarelli-Kohn-Nirenberg in \cite{CKN} (cf. issues on the (partial) regularity of Leray-Hopf weak solutions),       
\item[(ii)] Due to its simple structure, discretization methods could provide  new insight into (\ref{NS}) which is not visible from  purely analytical methods, 
\item[(iii)] Mathematical analysis on numerical methods  WITHOUT assuming existence of any exact solution is an effective approach to establish mathematical or computational theories of more complicated Navier-Stokes systems such as incompressible or compressible systems with other kinds of boundary condition,  a free surface, multiple phases, multiple species (flow of mixture), etc.    
\end{itemize}  
Related to (i) and (ii), we refer to \cite{Gallouet1} for another way to construct a Leray-Hopf weak solution, where it is proved that the MAC scheme applied to (\ref{NS}) ($\Omega$ is restricted to  parallelepipeds there) is convergent to a Leray-Hopf weak solution. In \cite{Guillod}, the uniqueness of Leray-Hopf solutions  is investigated through a numerical approach, though more substantial analysis is required to upgrade the result to a computer-assisted proof.  Related to (iii), we refer to \cite{Karper}, \cite{Feireisl2} and \cite{Gallouet2} for convergence proofs of numerical schemes applied to the compressible Navier-Stokes equations in the class of weak solutions and   \cite{Feireisl4} for that in the class of measure-valued solutions; \cite{Feireisl1} and \cite{Feireisl3} for numerical methods of the  full Navier-Stokes-Fourier system.

In Section 2, we give a discrete Helmholtz-Hodge decomposition operator. In Section 3, we discretize \eqref{NS} with the decomposition operator in an implicit form and prove unconditional  solvability of the discrete problem. In Section 4, we show a priori $L^2$-estimates. In Section 5, we discuss weak convergence of the difference solutions. In Section 6, we demonstrate a new interpolation inequality and  prove strong convergence by means of  weak convergence. In Section 7, we conclude the paper by proving that the difference solutions converge to a Leary-Hopf weak solution (up to subsequence). In Appendix, we show  continuous interpolation of a function defined on a grid and  a discrete Poincar\'e type inequality. 
\medskip

\noindent{\bf Acknowledgement.}  The second author is supported by JSPS Grant-in-aid for Young Scientists (B) \#15K21369 and JSPS Grant-in-aid for Young Scientists \#18K13443.
\setcounter{section}{1}
\setcounter{equation}{0}
\section{Discrete Helmholtz-Hodge decomposition}

Consider the grid $h\Z^3:=\{ (hz_1,hz_2,hz_3)\,|\,z_1,z_2,z_3\in\Z \}$ with the mesh size $h>0$. Let $e^1,e^2,e^3$ be the standard basis of $\R^3$. 
The boundary of  $B\subset h\Z^3$ is defined as $\partial B:=\{ x\in B\,|\,\{x\pm he^i\,|\,i=1,2,3\}\not\subset B \}$. 

Let $\Omega\subset$ be a bounded, open, connected subset of $\R^3$ with a Lipschitz boundary $\partial\Omega$. Set 
$$C_h(x):=[x_1-\frac{h}{2},x_1+\frac{h}{2})\times[x_2-\frac{h}{2},x_2+\frac{h}{2})\times[x_3-\frac{h}{2},x_3+\frac{h}{2}).$$
  Our discretization of (\ref{NS}) will take place on the set 
$$\Omega_h:=\{ x\in \Omega\cap h\Z^3\,|\,\,\,\,C_{4h}(x)\subset \Omega\}.$$ 
Define the discrete derivatives of a function $\phi:\Omega_h\to\R$ as
\begin{eqnarray*}
&&D_i^+\phi(x):=\frac{\phi(x+he^i)-\phi(x)}{h},\,\,\,D_i^-\phi(x):=\frac{\phi(x)-\phi(x-he^i)}{h},\\ 
&&D_i^2\phi(x):=\frac{\phi(x+he^i)+\phi(x-he^i)-2\phi(x)}{h^2}
\end{eqnarray*}  
for each $x\in \Omega_h$,  where {\it we always assume that $\phi$ is extended to be $0$ outside $\Omega_h$, i.e., $\phi(x+ he^i)=0$ (resp. $\phi(x- he^i)=0$) if $x+ he^i\not\in \Omega_h$ (resp. $\phi(x- he^i)=0$)}. For $x,y\in\R^d$, set $x\cdot y:=\sum_{i=1}^dx_iy_i$, $|x|:=\sqrt{x\cdot x}$. Define the discrete gradient  $\mathcal{D}$ and the discrete divergence $\mathcal{D}\cdot$ for  functions $\phi:\Omega_h\to\R$ and $w=(w_1,w_2,w_3):\Omega_h\to\R^3$ with $0$-extension as 
$$\D\phi(x):=(D_1^+\phi(x),D_2^+\phi(x),D_3^+\phi(x)),\,\,\,\D\cdot w(x):=D_1^-w_1(x)+D_2^-w_2(x)+D_3^-w_3(x)$$
for each $x\in \Omega_h$. Note that Chorin \cite{Chorin} uses the central difference $\frac{\phi(x+he^i)-\phi(x-he^i)}{2h}$ for the discrete gradient and divergence.     
\begin{Lemma}\label{vector-cal}
For $\phi:\Omega_h\to\R$ and $w:\Omega_h\to\R^3$ with $\phi|_{\partial \Omega_h}=0$ and $w|_{\partial \Omega_h}=0$, we have 
\begin{eqnarray*}
\sum_{x\in \Omega_h\setminus \partial \Omega_h}w(x) \cdot \D\phi(x)=-\sum_{x\in \Omega_h\setminus \partial\Omega_h}(\D\cdot w(x) ) \phi(x).
\end{eqnarray*}
\end{Lemma}
\begin{proof}
Due to the boundary condition of $w$ and $\phi$, we have   
\begin{eqnarray*}
\sum_{x\in \Omega_h\setminus \partial\Omega_h}w(x) \cdot \D\phi(x)&=&\sum_{x\in \Omega_h\setminus \partial\Omega_h}\sum_{i=1}^3w_i(x)\frac{\phi(x+he^i)-\phi(x)}{h}\\
&=&\sum_{x\in \Omega_h\setminus \partial\Omega_h}\sum_{i=1}^3\frac{1}{h}\big(w_i(x-he^i)\phi(x)-w_i(x)\phi(x)\big)\\
&=&-\sum_{x\in \Omega_h\setminus \partial\Omega_h}(\D\cdot w(x) ) \phi(x).
\end{eqnarray*}
\end{proof}
We consider decomposition of a function $u:\Omega_h\to\R^3$ of the form $u=w+\D\phi$ with $\D\cdot w=0$. In our decomposition, we ask the $0$-boundary condition also to $\phi$ instead of the $0$-mean condition.    
\begin{Thm}\label{Projection}
For each function  $u:\Omega_h\to\R^3$, there exist unique functions $w:\Omega_h\to\R^3$ and  $\phi:\Omega_h\to\R$ such that 
\begin{eqnarray}\label{HHD}
&&\D\cdot w=0,\quad 
w+\D \phi =u \mbox{ on $\Omega_h\setminus \partial\Omega_h$};\quad w=0,\quad  \phi=0 \mbox{ on $\partial\Omega_h$},
\end{eqnarray}
where $u$ does not necessarily need to vanish on $\partial \Omega_h$. 
\end{Thm}
\begin{proof}
Our argument will show how to construct $w$ and $\phi$. We label each point of $\Omega_h\setminus\partial\Omega_h$ as $ \Omega_h\setminus\partial\Omega_h=\{x^{1},x^{2},\ldots,x^{a}  \}$. Introduce $y,\alpha\in\R^{4a}$ as  
\begin{eqnarray*}
y&=&\big(w_1(x^{1}),\ldots,w_1(x^{a}),w_2(x^{1}),\ldots,w_2(x^{a}),w_3(x^{1}),\ldots,w_3(x^{a}), \phi(x^1),\ldots,\phi(x^{a})   \big),\\
\alpha&=&\big(0,\ldots,0,u_1(x^{1}),\ldots,u_1(x^{a}),u_2(x^{1}),\ldots,u_2(x^{a}),u_3(x^{1}),\ldots,u_3(x^{a})\big).
\end{eqnarray*}
Then, the equations $\D\cdot w=0$, $w+\D \phi =u \mbox{ on $\Omega_h\setminus \partial\Omega_h$}$ with the $0$-boundary condition of $w$ and $\phi$  give a $4a$-system of linear equations, which is denoted by $Ay=\alpha$ with a $4a\times4a$-matrix $A$. Note that $A$ is independent of $u$. Our assertion holds, if $A$ is invertible. To prove invertibility of $A$, we show that $Ay=0$ if and only if $y=0$. There is at least one $y$ satisfying $Ay=0$. Then, we obtain  at least one pair $w,\phi$ satisfying (\ref{HHD})$|_{u=0}$. By Lemma \ref{vector-cal}, we see that $\sum_{x\in\Omega_h\setminus\partial\Omega_h} w(x)\cdot \D\phi(x)=0$. Hence,  (\ref{HHD})$|_{u=0}$ yields 
\begin{eqnarray*}
\sum_{x\in\Omega_h\setminus\partial\Omega_h} |w(x)|^2=0,\quad \sum_{x\in\Omega_h\setminus\partial\Omega_h} |\D\phi(x)|^2=0.     
\end{eqnarray*}
Therefore, $w=0$ on $\Omega_h$ and $\D\phi=0$ on  $\Omega_h\setminus\partial\Omega_h$. The latter equality implies $\phi=0$ on $\Omega_h$ due to the $0$-boundary condition.  Thus, $A$ is invertible. 

Suppose that there are two pairs $w,\phi$ and $\tilde{w},\tilde{\phi}$ which satisfy (\ref{HHD}).  Then, we see that $w-\tilde{w}$, $\phi-\tilde{\phi}$ yields the unique trivial solution of $Ay=0$. Therefore, we conclude that $w=\tilde{w}$ and $\phi=\tilde{\phi}$.   \end{proof}
\medskip\medskip

\noindent{\bf Definition.} {\it Define the discrete Helmholtz-Hodge decomposition operator  $P_h$ for each function $u:\Omega_h\to\R^3$ as 
$$P_hu:=w\quad(\mbox{$w$ is the one mentioned in Theorem \ref{Projection}}).$$}
\begin{Thm}\label{22estimate}
We have the following estimates in regards to the decomposition $u=P_h u+\D \phi$: 
\begin{eqnarray*}
&&\sum_{x\in \Omega_h}|P_h u(x)|^2\le \sum_{x\in\Omega_h\setminus \partial\Omega_h}|u(x)|^2,\quad \sum_{x\in \Omega_h\setminus \partial\Omega_h}|\D\phi(x)|^2\le \sum_{x\in\Omega_h\setminus \partial\Omega_h}|u(x)|^2,\\
&&\sum_{x\in \Omega_h}|\phi(x)|^2\le A\sum_{x\in\Omega_h\setminus \partial\Omega_h}|\D\phi(x)|^2\le A\sum_{x\in\Omega_h\setminus \partial\Omega_h}|u(x)|^2,
\end{eqnarray*}
where $A>0$ is a constant depending only on $\Omega$. Furthermore, if $u|_{\partial\Omega_h}=0$, we have 
\begin{eqnarray}\label{242424242}
\sum_{x\in \Omega_h\setminus \partial\Omega_h}|u(x)-P_h u(x)|^2\le A\sum_{x\in\Omega_h\setminus \partial\Omega_h}|\D\cdot u(x)|^2. 
\end{eqnarray}
\end{Thm}
\begin{proof} 
By Lemma \ref{vector-cal}, we have 
\begin{eqnarray*}
\sum_{x\in\Omega_h\setminus\partial\Omega_h}|u(x)|^2&=&\sum_{x\in\Omega_h\setminus\partial\Omega_h}(P_hu(x)+\D\phi(x))\cdot (P_hu(x)+\D\phi(x))\\
&=&\sum_{x\in \Omega_h\setminus\partial\Omega_h}|P_hu(x)|^2+\sum_{x\in \Omega_h\setminus\partial\Omega_h}|\D\phi(x)|^2,
\end{eqnarray*}
which yields the first and second inequalities. The third one follows from the discrete Poincar\'e type inequality proved in Appendix 2. We prove the last one.  By Lemma \ref{vector-cal} and  the discrete Poincar\'e type inequality, we have 
\begin{eqnarray*}
&&\sum_{x\in \Omega_h\setminus \partial\Omega_h}|u(x)-P_h u(x)|^2
=\sum_{x\in \Omega_h\setminus \partial\Omega_h}|\D\phi(x)|^2\\
&&\qquad =\Big(  \sum_{x\in \Omega_h\setminus \partial\Omega_h}|u(x)-P_h u(x)|^2 \Big)^\2\Big(  \sum_{x\in \Omega_h\setminus \partial\Omega_h}|\D\phi(x)|^2 \Big)^\2 \\
&&\qquad=\sum_{x\in \Omega_h\setminus \partial\Omega_h} (u(x)-P_h u(x))\cdot \D \phi(x)\\
&&\qquad =-\sum_{x\in \Omega_h\setminus \partial\Omega_h} (\D\cdot u(x))\phi(x)\le \Big(\sum_{x\in \Omega_h\setminus \partial\Omega_h} |\D\cdot u(x)|^2\Big)^\2 \Big(\sum_{x\in \Omega_h\setminus \partial\Omega_h}|\phi(x)|^2\Big)^\2\\
&&\qquad\le \Big(\sum_{x\in \Omega_h\setminus \partial\Omega_h} |\D\cdot u(x)|^2\Big)^\2  \Big(A\sum_{x\in \Omega_h\setminus \partial\Omega_h}|\D\phi(x)|^2\Big)^\2.
\end{eqnarray*}
\end{proof}
\setcounter{section}{2}
\setcounter{equation}{0}
\section{Discrete problem}

We discretize (\ref{NS}) with the discretization parameter $\tau>0$ for time and $h>0$ for space on $\Omega_h$ introduced in Section 2. Let $T_\tau\in\N$ be such that $T\in[\tau T_\tau,\tau T_\tau+\tau)$. Let  $v^0=(v^0_1,v^0_2,v^0_3)\in L^2_{\sigma}(\Omega)$ and let  $f=(f_1,f_2,f_3)\in L^2_{loc}([0,\infty); L^2(\Omega)^3)$.     
Define $\tilde{u}^0=(\tilde{u}_1^0,\tilde{u}_2^0,\tilde{u}_3^0):\Omega_h\to\R^3$ and $f^{n+1}=(f^{n+1}_1,f^{n+1}_2,f^{n+1}_3):\Omega_h\to\R^3$, $n=0,1,\cdots,T_\tau-1$ as  
\begin{eqnarray*}
\tilde{u}^0_i(x)&:=&\left\{
\begin{array}{lll}
&\dis h^{-3}\int_{{C_h(x)}}v^0_i(y)dy,\quad x\in\Omega_h\setminus\partial\Omega_h,
\medskip\\
&0\mbox{\quad\quad\quad\,\,\,\quad\,\,\, otherwise},
\end{array}
\right. \\\\
f^{n+1}_i(x) &:=&\tau^{-1}h^{-3}\int_{\tau n}^{\tau (n+1)}\int_{{C_h(x)}}f_i(t,y)dydt, \quad   x\in\Omega_h,
\end{eqnarray*}
 For each $n=0,1,\ldots,T_\tau$,  define functions $u^n=(u^n_1,u^n_2,u_3^n):\Omega_h\to\R^3$ and    $u^{n+\2}=(u^{n+\2}_1,u^{n+\2}_2,u_3^{n+\2}):\Omega_h\to\R^3$ in the following manner:    
\begin{eqnarray}\label{initial-1}
u^0&=&P_h \tilde{u}^0,\\\label{fractional-1}
\frac{u^{n+\2}_i(x)-u^n_i(x)}{\tau}&=&-\sum_{j=1}^3 \frac{u^n_j(x-he^j)D_j^+u_i^{n+\2}(x-he^j)+u^n_j(x)D_j^+u_i^{n+\2}(x)}{2}\\\nonumber
&&\!\!\!\!\!\!\!\!\!\!\!\!\!\!\!\!\!\!\!\!\!\!\!\!\!\!+\sum_{j=1}^3D_j^2u_i^{n+\2} (x)+f^{n+1}_i(x),\quad x\in\Omega_h\setminus\partial\Omega_h,\quad i=1,2,3,\\\label{fractional-2}
u^{n+\2}_i(x)&=&0, \quad x\in\partial\Omega_h,\\\label{fractional-3}
u^{n+1}&=&P_hu^{n+\2},
\end{eqnarray}
where \eqref{initial-1}-\eqref{fractional-3} are recurrence equations in an implicit form.  As we will see below, our form of the nonlinear terms means a lot in a priori $L^2$-estimates of the nonlinear term. The presence of  $\pm he^j$ in the nonlinear terms is originally seen in the pioneering works by Ladyzhenskaya (see \cite{Ladyzhenskaya}). 

For functions $u,w:\Omega_h\to\R^3$ or $\R$, we define the discrete $L^2$-inner product and norm as 
$$(u,w)_{\Omega_h}:=\sum_{x\in\Omega_h}u(x)\cdot w(x) h^3,\quad \norm u \norm_{\Omega_h}:=\sqrt{(u,u)_{\Omega_h}}.$$
We prove unconditional solvability of the equations \eqref{fractional-1}-\eqref{fractional-2} with respect to $u^{n+\2}$. Chorin  \cite{Chorin} proved conditional solvability of his original scheme with a scale depending on the maximum norm of $u^n$. 
\begin{Thm}\label{implicit}
Suppose that $u^n:\Omega_h\to\R^3$ satisfies $\D\cdot u^n=0$ on $\Omega_h\setminus\partial\Omega_h$ and $u^n=0$ on $\partial\Omega_h$. Then,  the equation \eqref{fractional-1}-\eqref{fractional-2} is uniquely solvable with respect to $u^{n+\2}$ for any mesh size $h,\tau$. 
\end{Thm}
\begin{proof}
We label the elements of $\Omega_h\setminus\partial\Omega_h$ as $x^{1},x^{2},\ldots,x^{a}$. Introduce $y,\alpha\in\R^{3a}$ as  
\begin{eqnarray*}
y&:=&\big(u_1^{n+\2}(x^{1}),\ldots,u_1^{n+\2}(x^{a}),u_2^{n+\2}(x^{1}),\ldots,u_2^{n+\2}(x^{a}),u_3^{n+\2}(x^{1}),\ldots,u_3^{n+\2}(x^{a})   \big),\\
\alpha&:=&\big(u_1^{n}(x^{1})+\tau f^{n+1}_1(x^1),\ldots,u_1^{n}(x^{a})+\tau f^{n+1}_1(x^a),u_2^{n}(x^{1})+\tau f^{n+1}_2(x^1),\\
&&\ldots,u_2^{n}(x^{a})+\tau f^{n+1}_2(x^a),u_3^{n}(x^{1})+\tau f^{n+1}_3(x^1),\ldots,u_3^{n}(x^{a}) +\tau f^{n+1}_3(x^a)  \big).
\end{eqnarray*}
Then,  \eqref{fractional-1}-\eqref{fractional-2} are equivalent to the linear equations 
$$A(u^n;h,\tau)y=\alpha,$$
where $A(u^n;h,\tau)$  is a $3a\times3a$-matrix depending on $u^n,h,\tau$. 

We prove that the matrix $A(u^n;h,\tau)$ is invertible if $u^n$ satisfies $\D\cdot u_n=0$ in $\Omega_h\setminus\partial\Omega_h$. For this purpose, we check that  $A(u^n;h,\tau)y=0$ has the unique solution $y=0$.  Let $y=y_0$ be a solution of $A(u^n;h,\tau)y=0$. Then, we have $u^{n+\2}:\Omega_h\to\R^3$ with $u^{n+\2}|_{\partial\Omega_h}=0$ such that    
\begin{eqnarray*}
u^{n+\2}_i(x)&=&-\tau\sum_{j=1}^3 \frac{u^n_j(x-he^j)D_j^+u_i^{n+\2}(x-he^j)+u^n_j(x)D_j^+u_i^{n+\2}(x)}{2}\\\nonumber
&&+\tau\sum_{j=1}^3D_j^2u_i^{n+\2} (x),\quad x\in\Omega_h\setminus\partial\Omega_h,\quad i=1,2,3.
\end{eqnarray*}
Then, we have
\begin{eqnarray*}
(u^{n+\2},u^{n+\2})_{\Omega_h}&=&\norm u^{n+\2}\norm_{\Omega_h}^2\\
&=&-\frac{\tau}{2}\sum_{i,j=1}^3 \sum_{x\in\Omega_h\setminus\partial\Omega_h}\Big(u^n_j(x-he^j)D_j^+u^{n+\2}_i(x-he^j)\\
&& + u^n_j(x)D_j^+u^{n+\2}_i(x)\Big)u^{n+\2}_i(x)h^3\\
 &&+\tau\sum_{i,j=1}^3 \sum_{x\in\Omega_h\setminus\partial\Omega_h} D_j^2u_i^{n+\2}(x)u_i^{n+\2}(x) h^3.
\end{eqnarray*}
Here, the  above  two  summations are denoted by (i), (ii), respectively. Noting the $0$-boundary condition of $ u^{n+\2}$, we have 
 \begin{eqnarray*}
{\rm(i)}& =&\sum_{i,j=1}^3 \sum_{x\in\Omega_h\setminus\partial\Omega_h}\Big(u^n_j(x-he^j)\frac{u^{n+\2}_i(x)-u^{n+\2}_i(x-he^j)}{h}\\
&&+u^n_j(x)\frac{u^{n+\2}_i(x+he^j)-u^{n+\2}_i(x)}{h} \Big)u^{n+\2}_i(x)h^3\\
&=&\sum_{i,j=1}^3 \sum_{x\in\Omega_h\setminus\partial\Omega_h}
-\frac{u^n_j(x)-u^n_j(x-he^j)}{h}u^{n+\2}_i(x)^2h^3\\
&&+\sum_{i,j=1}^3 \sum_{x\in\Omega_h}
\frac{1}{h}u^n_j(x)u^{n+\2}_i(x+he^j)u^{n+\2}_i(x)h^3\\
&&-\sum_{i,j=1}^3 \sum_{x\in\Omega_h} \frac{1}{h}u^n_j(x-he^j)u^{n+\2}_i(x-he^j)u^{n+\2}_i(x)h^3.
\end{eqnarray*} 
Shifting $x$ to $x+he^j$ in the last summation, we  obtain 
 \begin{eqnarray*}
{\rm(i)}=-\sum_{x\in\Omega_h\setminus\partial\Omega_h}\big(\D\cdot u^n(x)\big)|u^{n+\2}(x)|^2h^3.
\end{eqnarray*} 
Similarly, we obtain 
\begin{eqnarray*}
\rm{(ii)}&=&  \sum_{i,j=1}^3 \sum_{x\in\Omega_h\setminus\partial\Omega_h}
\frac{u_i^{n+\2} (x+he^j)-2u_i^{n+\2} (x)+u_i^{n+\2} (x-he^j)}{h^2}u^{n+\2}_i(x)h^3\\
&=& \sum_{i,j=1}^3 \sum_{x\in\Omega_h\setminus\partial\Omega_h}
\frac{u_i^{n+\2} (x+he^j)-u_i^{n+\2} (x)}{h^2}u^{n+\2}_i(x)h^3\\
&&- \sum_{i,j=1}^3 \sum_{x\in\Omega_h\setminus\partial\Omega_h}\frac{u_i^{n+\2} (x)-u_i^{n+\2} (x-he^j)}{h^2}u^{n+\2}_i(x)h^3\\
&=& \sum_{i,j=1}^3 \sum_{x\in\Omega_h}
\frac{u_i^{n+\2} (x+he^j)-u_i^{n+\2} (x)}{h^2}u^{n+\2}_i(x)h^3\\
&&- \sum_{i,j=1}^3 \sum_{x\in\Omega_h}\frac{u_i^{n+\2} (x+he^j)-u_i^{n+\2} (x)}{h^2}u^{n+\2}_i(x+he^j)h^3\\
&=&- \sum_{j=1}^3\norm D^+_ju^{n+\2}\norm_{\Omega_h}^2\le0. 
\end{eqnarray*} 
Hence, the discrete divergence free constraint of $u^{n}$ implies 
$$\norm u^{n+\2}\norm_{\Omega_h}^2+\tau\sum_{j=1}^3\norm D^+_ju^{n+\2}\norm_{\Omega_h}^2=0.$$
Thus, we conclude that $u^{n+\2}=0$ and $y_0=0$. \end{proof}
\setcounter{section}{3}
\setcounter{equation}{0}
\section{$L^2$-estimate}

In this section, we obtain several $L^2$-estimates. 
Recall that $v^0$ is taken from $L^2_{\sigma}(\Omega)$ and $f$ from $L^2_{loc}([0,\infty);L^2(\Omega)^3)$. We observe that 
\begin{eqnarray*}
&&\tilde{u}^0_i(x)^2h^3=\Big(h^{-3}\int_{{C_h(x)}}v^0_i(y)dy\Big)^2h^3
\le h^{-3}\Big\{ \Big(\int_{{C_h(x)}}1dy \Big)^\2   \Big(\int_{{C_h(x)}}v^0_i(y)^2dy\Big)^\frac{1}{2}   \Big\}^2\\
&&\qquad\qquad=\int_{{C_h(x)}}v^0_i(y)^2dy.
\end{eqnarray*}
Hence, with Theorem \ref{22estimate}, we see that   $u^0=P_h\tilde{u}^0$ satisfies 
\begin{eqnarray*}\label{initial-L2}
\norm u^0\norm_{\Omega_h}\le\norm \tilde{u}^0\norm_{\Omega_h}\le \norm v^0\norm_{L^2(\Omega)^3}.
\end{eqnarray*}
Similar calculation yields 
\begin{eqnarray*}\label{L2L2}
&&\sum_{m=0}^n\norm f^{m+1}\norm_{\Omega_h}^2\tau \le \norm f\norm_{L^2([0,\tau (n+1)];L^2(\Omega)^3)}^2\le \norm f\norm_{L^2([0,T];L^2(\Omega)^3)}^2,\quad 0\le n<T_\tau
\end{eqnarray*}
\begin{Thm}\label{solvability}
The discrete problem \eqref{initial-1}-\eqref{fractional-3} is uniquely solvable with the  following estimates for $n=0,1,\ldots,T_\tau-1$: 
\begin{eqnarray} \label{411}
\quad \norm u^{n+\2}\norm_{\Omega_h}&\le& \norm u^n\norm_{\Omega_h}+\norm f^{n+1}\norm_{\Omega_h}\tau,\\\label{412}
\norm u^{n+1}\norm_{\Omega_h}&\le& \norm u^0\norm_{\Omega_h}+\sum_{m=0}^{T_\tau}\norm f^{m+1}\norm_{\Omega_h}\tau\\\nonumber 
&\le&  \norm v^0\norm_{L^2(\Omega)^3}+ \sqrt{T}\norm f\norm_{L^2([0,T];L^2(\Omega)^3)}   ,\\\label{413}
\norm u^{n+1}\norm_{\Omega_h}^2&\le&\norm u^0\norm_{\Omega_h}^2-\sum_{m=0}^n  \Big(  \sum_{j=1}^3\norm D^+_ju^{m+\2}\norm_{\Omega_h}^2 \Big)\tau \\
\nonumber 
&& +2\sum_{m=0}^n\norm u^m\norm_{\Omega_h} 
 \norm f^{m+1} \norm_{\Omega_h}\tau +\sum_{m=0}^{n}\norm f^{m+1}\norm_{\Omega_h}^2\tau^2.
\end{eqnarray}
\end{Thm}
\begin{proof}
By Theorem \ref{implicit} and (i)-(ii) in its proof,  \eqref{initial-1}-\eqref{fractional-3} is uniquely solvable with  
\begin{eqnarray*}
\norm u^{n+\2}\norm_{\Omega_h}^2 &\le& ( u^{n}, u^{n+\2})_{\Omega_h}-\sum_{j=1}^3\norm D^+_ju^{n+\2}\norm_{\Omega_h}^2\tau +(f^{n+1},u^{n+\2})_{\Omega_h}\tau \\
&&\le \norm u^{n+1}\norm_{\Omega_h}\norm u^{n+\2}\norm_{\Omega_h}+\norm f^{n+1}\norm_{\Omega_h}\norm u^{n+\2}\norm_{\Omega_h}\tau.
\end{eqnarray*} 
Hence, we obtain 
\begin{eqnarray*}
\norm u^{n+\2}\norm_{\Omega_h}&\le&\norm u^{n}\norm_{\Omega_h}+\norm f^{n+1}\norm_{\Omega_h}\tau,\\
\norm u^{n+1}\norm_{\Omega_h}&=& \norm P_hu^{n+\2}\norm_{\Omega_h}\le  \norm u^{n+\2}\norm_{\Omega_h} 
\le\norm u^{n} \norm_{\Omega_h}+ \norm f^{n+1}\norm_{\Omega_h}\tau,\\
\norm u^{n+1}\norm_{\Omega_h}^2&\le& \norm u^{n+\2} \norm_{\Omega_h}^2\\
&\le&  \norm u^{n} \norm_{\Omega_h}\norm u^{n+\2} \norm_{\Omega_h} -\sum_{j=1}^3\norm D^+_ju^{n+\2}\norm_{\Omega_h}^2\tau+ \norm f^{n+1}\norm_{\Omega_h}\norm u^{n+\2}\norm_{\Omega_h}\tau   \\
&\le& \norm u^{n} \norm_{\Omega_h}^2 -\sum_{j=1}^3\norm D^+_ju^{n+\2}\norm_{\Omega_h}^2\tau+2\norm u^{n} \norm_{\Omega_h}\norm f^{n+1} \norm_{\Omega_h}\tau\\
&&+\norm f^{n+1}\norm_{\Omega_h}^2\tau^2. 
\end{eqnarray*}
These inequalities imply \eqref{411}-\eqref{413}.  
 \end{proof}
\setcounter{section}{4}
\setcounter{equation}{0}
\section{Weak convergence}

 Set $\delta=(h,\tau)$ and 
 $$C_h^+(x):=C_h(x+\frac{h}{2}e^1+\frac{h}{2}e^2+\frac{h}{2}e^3)=[x_1,x_1+h)\times[x_2,x_2+h)\times[x_3,x_3+h).$$
For the solution $u^n,u^{n+\2}$ of  \eqref{initial-1}-\eqref{fractional-3}, define the step functions  $u_\delta:[0,T]\times\Omega\to\R^3$, $v_\delta,w^i_\delta:[0,T]\times\Omega\to\R^3$, $i=1,2,3$ as 
\begin{eqnarray*}
u_\delta(t,x)&:=&\left\{
\begin{array}{lll}
&u^{n}(y)\mbox{\quad\quad\ for $t\in[n\tau,n\tau+\tau)$, $x\in {C_h^+(y)}$, $y\in \Omega_h$},
\medskip\\
&0\mbox{\quad\quad\quad\,\,\,\quad\,\,\, otherwise},
\end{array}
\right. \\
v_\delta(t,x)&:=&\left\{
\begin{array}{lll}
&u^{n+\2}(y)\mbox{\quad\quad\ for $t\in[n\tau,n\tau+\tau)$, $x\in {C_h^+(y)}$, $y\in \Omega_h$},
\medskip\\
&0\mbox{\quad\quad\quad\,\,\,\quad\,\,\, otherwise},
\end{array}
\right. \\
w^i_\delta(t,x)&:=&\left\{
\begin{array}{lll}
&D_i^+u^{n+\2}(y)\mbox{\quad for $t\in[n\tau,n\tau+\tau)$, $x\in {C_h^+(y)}$, $y\in \Omega_h$},
\medskip\\
&0\mbox{\quad\quad\quad\quad\quad\, otherwise}, 
\end{array}
\right.
\end{eqnarray*}
where $n=0,1,\ldots,T_\tau$.  In the rest of our argument, the statement ``there exists a sequence $\delta\to0$ ...'' means ``there exists a sequence $\delta_l=(h_l,\tau_l)$ with $h_l,\tau_l\searrow0$ as $l\to\infty$ ...''. The next theorem states weak convergence of the above step functions.   
\begin{Thm}\label{weak convergence}
There exists a sequence $\delta\to0$ and a function $v\in L^2([0,T];H^1_{0,\sigma}(\Omega))$ for which  the following weak convergence holds: 
\begin{eqnarray}\label{511}
&&u_\delta \wto v \mbox{ \quad in $L^2([0,T];L^2(\Omega)^3)$ as $\delta\to0$},\\\label{512}
&&v_\delta \wto v \mbox{ \quad in $L^2([0,T];L^2(\Omega)^3)$ as $\delta\to0$},\\\label{523}
&&w^i_\delta \wto \partial_{x_i} v \mbox{ \quad in $L^2([0,T];L^2(\Omega)^3)$ as $\delta\to0$ ($i=1,2,3$)}.
\end{eqnarray}  
\end{Thm}
\begin{proof}
Due to Theorem \ref{solvability}, $\{u_\delta{}_j\}$, $\{v_\delta{}_j\}$, $\{w^i_\delta{}_j\}$ ($i,j=1,2,3$) are bounded in the Hilbert space $L^2([0,T];L^2(\Omega))$. Hence, there exists a sequence $\delta\to0$ and functions  $u=(u_1,u_2,u_3),v=(v_1,v_2,v_3),w^i=(w_1^i,w_2^i,w_3^i)\in L^2([0,T];L^2(\Omega)^3)$ such that 
$$u_\delta{}_j \wto u_j,\quad v_\delta{}_j \wto v_j,\quad w^i_\delta{}_j \wto w^i_j \mbox{\quad in $L^2([0,T];L^2(\Omega)^3)$ as $\delta\to0$\quad ($i,j=1,2,3$)}.$$
\indent We prove $u=v$. For each $\phi\in C^\infty([0,T];C^\infty_0(\Omega)^3)$,  we have 
\begin{eqnarray*}
(u-v,\phi)_{L^2([0,T];L^2(\Omega)^3)}&=&(u-u_\delta,\phi)_{L^2([0,T];L^2(\Omega)^3)}+(u_\delta-v_\delta,\phi)_{L^2([0,T];L^2(\Omega)^3)}\\
&&+(v_\delta-v,\phi)_{L^2([0,T];L^2(\Omega)^3)}.
\end{eqnarray*}
The first and third terms go to $0$ as $\delta\to0$. Hence, we may conclude $u=v$ by proving $(u_\delta-v_\delta,\phi)_{L^2([0,T];L^2(\Omega)^3)}\to 0$ as $\delta\to0$.
In fact, setting $\phi^n(\cdot):=\phi(\tau n,\cdot)$ and noting that $\phi(t,\cdot)$ is compactly supported in $\Omega$ (therefore, $\phi(t,\cdot)=0$ on $\partial\Omega_h$ for sufficiently small $h>0$), we have 
\begin{eqnarray*}
&&|(u_\delta-v_\delta,\phi)_{L^2([0,T];L^2(\Omega)^3)}|\le \sum_{n=0}^{T_\tau-1}|(u^{n+\2}-u^n,\phi^n)_{\Omega_h}|\tau+O(h),\\
&&\sum_{n=0}^{T_\tau-1}|(u^{n+\2}-u^n,\phi^n)_{\Omega_h}|\tau\\
&&\quad=\sum_{n=0}^{T_\tau-1}\Big|-\frac{\tau}{2} \sum_{j=1}^3 \Big(u^n_j(\cdot-he^j)D_j^+u^{n+\2}(\cdot-he^j)  +u^n_j(\cdot)D_j^+u^{n+\2}(\cdot), \phi^n\Big)_{\Omega_h}\\
&&\quad+\tau\sum_{j=1}^3  (D_j^2u^{n+\2},\phi^n)_{\Omega_h}+\tau(f^{n+1},\phi^n)_{\Omega_h}   \Big|\tau\\
&&\le \frac{\tau}{2}\max_{[0,T]\times\Omega,i} |\phi_i(t,x)|   
\sum_{n=0}^{T_\tau-1}\sum_{i,j=1}^3(\norm u^n_j\norm_{\Omega_h}^2+\norm D_j^+u_i^{n+\2}\norm_{\Omega_h}^2)\tau \\
&&+\frac{\tau}{2}\sum_{n=0}^{T_\tau-1}\sum_{i,j=1}^3 (\norm D^+_ju_i^{n+\2}\norm_{\Omega_h}^2+\norm D_j^+\phi_i^n\norm_{\Omega_h}^2)\tau 
+\frac{\tau}{2} \sum_{n=0}^{T_\tau-1}  (\norm f^{n+1}\norm_{\Omega_h}^2+\norm \phi^n\norm_{\Omega_h}^2)\tau   \\
&&=O(\tau) \to0\quad \mbox{ as $\delta\to0$},
\end{eqnarray*}
due to Theorem \ref{solvability}.  

We prove $\partial_{x_i} v=w^i$.  
For each $\phi\in C^\infty_0((0,T);C^\infty_0(\Omega))$, setting $\phi^n(\cdot):=\phi(\tau n,\cdot)$, we observe that 
\begin{eqnarray*}
&&\sum_{y\in\Omega_h}D_i^+u^{n+\2}_j(y)\phi^n(y) h^3= \sum_{y\in\Omega_h}\frac{u^{n+\2}_j(y+he^i)-u^{n+\2}_j(y)}{h}\phi^n(y) h^3\\
&&\qquad= \sum_{y\in\Omega_h}\frac{u^{n+\2}_j(y+he^i)\phi^n(y)-u^{n+\2}_j(y+he^i)\phi^n(y+he^i)}{h} h^3\\
&&\qquad=-\sum_{y\in\Omega_h}u^{n+\2}_j(y+he^i)D_i^+\phi^n(y) h^3\\
&&\qquad=-\sum_{y\in\Omega_h}u^{n+\2}_j(y)D_i^+\phi^n(y-he^i) h^3. 
\end{eqnarray*} 
 Hence, noting the regularity of $\phi$ and Theorem \ref{solvability}, we have 
\begin{eqnarray*}
(w^i_\delta{}_j,\phi)_{L^2([0,T];L^2(\Omega))}  &=& \sum_{n=0}^{T_\tau-1}\sum_{y\in\Omega_h}D_i^+u^{n+\2}_j(y)(\phi^n(y)+O(h)) h^3\tau\\
&=&\sum_{n=0}^{T_\tau-1}\sum_{y\in\Omega_h}D_i^+u^{n+\2}_j(y)\phi^n(y)h^3\tau+O(h)\\
&=&-\sum_{n=0}^{T_\tau-1}\sum_{y\in\Omega_h} u^{n+\2}_j(y)D_i^+\phi^n(y) h^3\tau+O(h),\\
(v_{\delta}{}_j,\partial_{x_i}\phi)_{L^2([0,T];L^2(\Omega))}  &=& \sum_{n=0}^{T_\tau-1}\sum_{y\in\Omega_h}u^{n+\2}_j(y)(D_i^+\phi^n(y)+O(h)) h^3\tau\\
&=&\sum_{n=0}^{T_\tau-1}\sum_{y\in\Omega_h}u^{n+\2}_j(y)D_i^+\phi^n(y) h^3\tau+O(h).  
\end{eqnarray*} 
Therefore, the weak convergence implies $(v_j,\partial_{x_i}\phi)_{L^2([0,T];L^2(\Omega))}=-(w^i_j,\phi)_{L^2([0,T];L^2(\Omega))}$ for any $\phi\in C^\infty_0((0,T);C^\infty_0(\Omega))$.  

We prove $\nabla\cdot v(t,\cdot)=0$ a.e. $t\in[0,T]$ in the sense of the weak derivative. For each $\phi\in C^\infty_0((0,T);C^\infty_0(\Omega))$, we have $\phi=0$ on a neighborhood of  $\partial\Omega$ for all sufficiently small $h>0$. Hence, with Lemma \ref{vector-cal} and $\phi^n(\cdot):=\phi(\tau n,\cdot)$, we have   
\begin{eqnarray*}
0&=&\sum_{n=0}^{T_\tau-1}\sum_{y\in\Omega_h}\D\cdot u^{n}(y)\phi^n(y) h^3\tau=-\sum_{n=0}^{T_\tau-1}\sum_{y\in\Omega_h}  u^n(y) \cdot \D\phi^n(y)  h^3\tau\\
&=&-\sum_{i=1}^3 (u_{\delta}{}_i,\partial_{x_i}\phi)_{L^2([0,T];L^2(\Omega))} +O(h)
\to- \sum_{i=1}^3(v_i,\partial_{x_i}\phi)_{L^2([0,T];L^2(\Omega))}\mbox{\quad as $\delta\to0$}.
\end{eqnarray*} 
Therefore, we obtain $(\nabla\cdot v,\phi )_{L^2([0,T];L^2(\Omega))}=- \sum_{i=1}^3(v_i,\partial_{x_i}\phi)_{L^2([0,T];L^2(\Omega))}=0$ for any $\phi\in C^\infty_0((0,T);C^\infty_0(\Omega))$. Up to now, we proved $v\in L^2([0,T];H^1(\Omega)^3)$ and $\nabla\cdot v=0$ a.e. $t\in[0,T]$. 

We prove $v\in L^2([0,T];H^1_0(\Omega)^3)$. Let $\tilde{v}^{n}_\delta:\Omega\to\R^3$ be the Lipschitz interpolation of $u^{n+\2}$ by means of Appendix (1) and let $\tilde{v}_\delta:[0,T]\times\Omega\to\R^3$ be defined as $\tilde{v}_\delta(t,\cdot):=\tilde{v}_\delta^{n}(\cdot)$ for $t\in[\tau n,\tau n+\tau)\cap[0,T]$. Note that 
$$\tilde{v}_\delta\in L^2([0,T];H^1_0(\Omega)^3),\,\,\,\norm \tilde{v}_\delta-v_\delta\norm_{L^2([0,T];L^2(\Omega)^3)}=O(h),\,\,\, \norm \partial_{x_i}\tilde{v}_\delta\norm_{L^2([0,T];L^2(\Omega)^3)}\le K$$
for $i=1,2,3$, where $K$ is a constant independent from $\delta$. Hence, taking a subsequence if necessary, we see that $\tilde{v}_\delta\wto v$ in $L^2([0,T];L^2(\Omega)^3)$  as $\delta\to0$ and that there exists $\tilde{w}^i\in L^2([0,T];L^2(\Omega)^3)$ such that $\partial_{x_i}\tilde{v}_\delta\wto \tilde{w}^i$ in $L^2([0,T];L^2(\Omega)^3)$  as $\delta\to0$ for $i=1,2,3$. Since 
$(\partial_{x_i}\tilde{v}_{\delta j},\phi)_{L^2([0,T];L^2(\Omega))}=-(\tilde{v}_{\delta j},\partial_{x_i}\phi)_{L^2([0,T];L^2(\Omega))}$ for any $\phi\in C^\infty((0,T]);C^\infty_0(\Omega))$, we have 
$  (\tilde{w}^i_j,\phi)_{L^2([0,T];L^2(\Omega))}=-(v_j,\partial_{x_i}\phi)_{L^2([0,T];L^2(\Omega))}$ and $\tilde{w}^i=\partial_{x_i}v$. In particular, 
$$(\tilde{v}_\delta,\phi)_{L^2([0,T];H^1(\Omega)^3)}\to(v,\phi)_{L^2([0,T];H^1(\Omega)^3)} \mbox{ as $\delta\to0$ \,\,\, for any $\phi\in L^2([0,T];H^1(\Omega)^3)$.}$$
Since $\{\tilde{v}_\delta\}$ is a bounded sequence of the Hilbert space $L^2([0,T];H^1_0(\Omega)^3)$, taking a subsequence if necessary, we find $\tilde{v}\in  L^2([0,T];H^1_0(\Omega)^3)$ to which $\tilde{v}_\delta$ weakly converges in $L^2([0,T];H^1_0(\Omega)^3)$ as $\delta\to0$, i.e., 
$$(\tilde{v}_\delta,\phi)_{L^2([0,T];H^1(\Omega)^3)}\to(\tilde{v},\phi)_{L^2([0,T];H^1(\Omega)^3)} \mbox{ as $\delta\to0$ \,\,\, for any $\phi\in L^2([0,T];H^1_0(\Omega)^3)$.}$$  
Therefore, we have 
$(v-\tilde{v},\phi)_{L^2([0,T];H^1(\Omega)^3)}=0\mbox{ \,\,\,for any $\phi\in L^2([0,T];H^1_0(\Omega)^3)$.}$ 
Since $\tilde{v}_\delta-\tilde{v}\in  L^2([0,T];H^1_0(\Omega)^3)$, we obtain 
\begin{eqnarray*}
0&=&(v-\tilde{v},\tilde{v}_\delta-\tilde{v})_{L^2([0,T];H^1(\Omega)^3)}\\
&=&(v-\tilde{v},v-\tilde{v})_{L^2([0,T];H^1(\Omega)^3)}+(v-\tilde{v},\tilde{v}_\delta-v)_{L^2([0,T];H^1(\Omega)^3)}\\
&\to&\norm v-\tilde{v}\norm_{L^2([0,T];H^1(\Omega)^3)}^2\mbox{\quad as $\delta\to0$},
\end{eqnarray*}
which concludes that $v=\tilde{v}\in  L^2([0,T];H^1_0(\Omega)^3)$. 

Thus, with Lemma \ref{1key-lemma}, we see that the limit function $v$ belongs to $L^2([0,T];\tilde{H}^1_{0,\sigma}(\Omega))=L^2([0,T];H^1_{0,\sigma}(\Omega))$.  
\end{proof}
\setcounter{section}{5}
\setcounter{equation}{0}
\section{Strong convergence}
From now on, we always assume the following scaling condition:
\begin{eqnarray}\label{scaling}
\delta=(h,\tau)\to 0 \mbox{ \,\,\, together with $\dis h^{3-\alpha}\le \tau$},
\end{eqnarray}
where  $\alpha\in(0,2]$ is any constant.
We  prove that {\it the weak convergence of $\{v_\delta\}$ in Theorem \ref{weak convergence} is actually strong one}. The idea of our proof is the following: 
\begin{itemize}
\item[(S1)] Suppose that the weakly convergent sequence $\{ v_\delta \}$ obtained in  in Theorem \ref{weak convergence}, which is re-denoted by $\{v_\alpha\}_{\alpha\in\N}$,  is not strongly convergent in $L^2([0,T];L^2(\Omega)^3)$, i.e.,  $\{ v_\alpha \}$ is not a Cauchy sequence in $L^2([0,T];L^2(\Omega)^3)$.
\item[(S2)] Then, there exists $\ep_0>0$ such that for each $m\in\N$ we have $\alpha^1(m),\alpha^2(m)\ge m$ for which 
$0<\ep_0\le \norm v_{\alpha^1(m)}-v_{\alpha^2(m)} \norm_{L^2([0,T];L^2(\Omega)^3)}$
 holds.
\item[(S3)] We will see that $\norm v_{\alpha^1(m)}-v_{\alpha^2(m)}  \norm_{L^2([0,T];L^2(\Omega)^3)}$ is bounded from the above by two different kinds of ``norms''.
\item[(S4)] We are able to estimate  the  ``norms'' to tend to $0$ as $m\to\infty$, only with information on the weak convergence of $\{v_\delta\}$, and we reach a contradiction. 
\end{itemize}    
Note that we do not know if $\{u_\delta\}$ converges to $v$ strongly or not due to the absence of estimates for $D_j^+u^n$.  In order to carry out our idea, we first study an interpolation inequality of a sequence of the sum of two step functions, as well as some inequality obtained from \eqref{fractional-1}. 

Observe that, if $x\in\Omega_h\setminus\partial\Omega_h$ is such that $x\pm he_i\in\Omega_h\setminus\partial\Omega_h$, we have by  \eqref{fractional-1} and the discrete divergence free constraint of $u^n$,  
\begin{eqnarray*}
\D\cdot u^{n+\2}(x)&=&\D\cdot\Big\{-\tau\sum_{j=1}^3 \frac{u^n_j(x-he^j)D_j^+u^{n+\2}(x-he^j)+u^n_j(x)D_j^+u^{n+\2}(x)}{2}\\
&&+\tau\sum_{j=1}^3D_j^2u^{n+\2}(x)   +\tau f^{n+1}(x)\Big\}.
\end{eqnarray*}
Hence, we obtain for each $\phi\in C^\infty_0(\Omega)$ and sufficiently small $h>0$, 
\begin{eqnarray}\nonumber 
&&\Big(\D\cdot \Big(\sum_{j=1}^3D_j^2u^{n+\2}\Big),\phi\Big)_{\Omega_h}=-\Big(\sum_{j=1}^3D_j^2u^{n+\2},\D\phi\Big)_{\Omega_h}=\sum_{j=1}^3(D_j^+u^{n+\2},D_j^+\D\phi)_{\Omega_h},\\
\label{6-key}
&&|(\D\cdot u^{n+\2},\phi)_{\Omega_h}|\le \tau\max_{x\in\Omega} |\nabla \phi | \sum_{j=1}^3 \norm u_j^n \norm_{\Omega_h}\norm D_j^+u^{n+\2} \norm_{\Omega_h}\\\nonumber
&&\!\!\!\!\!\!\!\!\!\!\qquad\qquad\qquad+\tau \norm \D \phi \norm_{\Omega_h} \norm f^{n+1} \norm_{\Omega_h} +\tau\sum_{j=1}^3 \norm D_j^+\D \phi \norm_{\Omega_h}\norm D_j^+u^{n+\2}\norm_{\Omega_h}.
\end{eqnarray}
This estimate implies that normalized $u^{n+\2}$ is asymptotically divergence free in the proof of Lemma \ref{key-lemma} below. 
 
We proceed from (S2).  Let $u^n_{\alpha^i(m)},u^{n+\2}_{\alpha^i(m)}, f^{n+1}_{\alpha^i(m)},h_{\alpha^i(m)},\tau_{\alpha^i(m)}$, etc., be the quantities that yield $v_{\alpha^i(m)}$ ($i=1,2$).  Fix any $t\in[0,T]$ and let $n_{\alpha^i(m)}\in\N$ be such that 
$t\in[\tau_{\alpha^i(m)}n_{\alpha^i(m)},\tau_{\alpha^i(m)}n_{\alpha^i(m)}+\tau_{\alpha^i(m)})$. 
Introduce the following notation: 
\begin{eqnarray*}
&& \nnorm v_{\alpha^i(m)} (t,\cdot)\nnorm:= \Big(\norm u_{\alpha^i(m)}^{n_{\alpha^i(m)}+\2}\norm_{\Omega_{h_{\alpha^i(m)}}}^2+\sum_{j=1}^3\norm D^+_j u_{\alpha^i(m)}^{n_{\alpha^i(m)}+\2}\norm_{\Omega_{h_{\alpha^i(m)}}}^2\\
&&\quad +  \tau_{\alpha^i(m)}\norm f_{\alpha^i(m)}^{n_{\alpha^i(m)}+1}\norm^2_{\Omega_{h_{\alpha^i(m)}}} \Big)^\2, \,\,\,i=1,2,\\
&&\nnorm v_{\alpha^1(m)}(t,\cdot)-v_{\alpha^2(m)}(t,\cdot)\nnorm_{op} 
 \\
 &&\quad := \sup_{\substack{
 \phi\in  C^4_{0,\sigma}(\Omega), \\\\
 \norm \phi \norm_{W^{4,\infty}(\Omega)^3} =1,\\\\ 
  \phi\equiv 0\,\,{\rm on}\,\,C_{2h_{\alpha^i(m)}}(y)\\\\
  {\rm  for \,\,all}\,\,y\in\partial \Omega_{h_{\alpha^i(m)}},\\\\
   i=1,2       }  }   
 \Big|(u_{\alpha^1(m)}^{n_{\alpha^1(m)}+\2},  Q_{h_{\alpha^1(m)}}\phi)_{\Omega_{h_{\alpha^1(m)}}}-(u_{\alpha^2(m)}^{n_{\alpha^2(m)}+\2},  Q_{h_{\alpha^2(m)}}\phi)_{\Omega_{h_{\alpha^2(m)}}}\Big|,\\
 &&Q_h\phi:=\big( (Q_h\phi)_1,(Q_h\phi)_2,(Q_h\phi)_3  \big) \mbox{\quad for $\phi:\Omega\to\R^3$},\\
&&(Q_h\phi)_j:=\phi_j+\frac{h}{2}\frac{\partial \phi_j}{\partial x_j}+\frac{h^2}{12}\frac{\partial^2 \phi_j}{\partial x_j^2} . 
\end{eqnarray*}
It follows from the Taylor expansion that for each $\phi\in  C^4_{0,\sigma}(\Omega)$ there exists a constant $\beta>0$ such that 
\begin{eqnarray}\label{66Taylor}
|\D\cdot(Q_h\phi) (x)|\le \beta h^3 \mbox{ for all $x\in\Omega_h$}.
\end{eqnarray} 
\begin{Lemma}\label{key-lemma}
For each $\eta>0$, there exists $A_\eta>0$  independent of  $t\in[0,T]$ such that 
\begin{eqnarray}\label{key}
&& \norm v_{\alpha^1(m)}(t,\cdot)-v_{\alpha^2(m)}(t,\cdot)\norm_{L^2(\Omega)^3}\le \eta (\nnorm v_{\alpha^1(m)}(t,\cdot)\nnorm+\nnorm v_{\alpha^2(m)}(t,\cdot)\nnorm)\\\nonumber
&&\qquad\qquad\qquad\qquad+A_\eta \nnorm v_{\alpha^1(m)}(t,\cdot)-v_{\alpha^2(m)}(t,\cdot)\nnorm_{op} \quad \mbox{for all $m\in\N$ and $t\in[0,T]$}.
\end{eqnarray} 
\end{Lemma}
\begin{proof}
First we find $A_\eta$ for each $t\in[0,T]$.  
Set 
$$w^i_m=w^i_m(t):= v_{\alpha^i(m)}(t,\cdot).$$ 
Suppose that the assertion does not hold. Then, there exists some constant $\eta_0>0$ such that  for each $l\in\N$ we can find $m=m(l)\in\N$ such that $m(l)\nearrow\infty$ as $l\to\infty$  and 
\begin{eqnarray}\label{6161}
\norm w^1_{m(l)}-w^2_{m(l)} \norm_{L^2(\Omega)^3}> \eta_0 (\nnorm w_{m(l)}^1\nnorm+\nnorm w_{m(l)}^2\nnorm)+l \nnorm w^1_{m(l)}-w^2_{m(l)}\nnorm_{op}.
\end{eqnarray}
Set 
\begin{eqnarray*}
\tilde{w}_{m(l)}^i:=\frac{w^i_{m(l)}}{\nnorm w_{m(l)}^1\nnorm+\nnorm w_{m(l)}^2\nnorm}=\frac{u_{\alpha^i(m(l))}^{n_{\alpha^i(m(l))}+\2}}{\nnorm w_{m(l)}^1\nnorm+\nnorm w_{m(l)}^2\nnorm}:\Omega_{h_{\alpha^i(m(l))}}\to\R^3.
\end{eqnarray*}
We have 
\begin{eqnarray*}
&& \norm \tilde{w}_{m(l)}^i\norm_{\Omega_{h_{\alpha^i(m(l))}}}\le1,\quad 
\norm D_j^+ \tilde{w}^i_{m(l)}\norm_{\Omega_{h_{\alpha^i(m(l))}}}\le1,\quad j=1,2,3,\quad i=1,2.
\end{eqnarray*}
Let $\omega^i_{m(l)}:\Omega\to\R^3$ be the step function generated by $\tilde{w}_{m(l)}^i$ as 
\begin{eqnarray*}
\omega^i_{m(l)}(x):=\left\{
\begin{array}{lll}
&\tilde{w}^i_{m(l)}(y)\mbox{\quad\quad\ for  $x\in {C_{h_{\alpha^i(m(l))}}^+(y)}$, $y\in \Omega_{h_{\alpha^i(m(l))}}$},
\medskip\\
&0\mbox{\quad\quad\quad\,\,\,\quad\,\,\, otherwise}.
\end{array}
\right. 
\end{eqnarray*} 
Let $\bar{\omega}^i_{m(l)}:\Omega\to\R^3$ be the Lipschitz interpolation of $\omega^i_{m(l)}$ by means of Appendix (1).  We have 
\begin{eqnarray}\label{616161}
&& \norm \bar{\omega}^i_{m(l)}-\omega^i_{m(l)}\norm_{L^2(\Omega)^3}=O(h_{\alpha^i(m(l))}),\\\nonumber
&&\norm \partial_{x_j} \bar{\omega}^i_{m(l)}\norm_{L^2(\Omega)^3} \le K \mbox{ \quad for all $l\in\N$, $i=1,2$, $j=1,2,3$},
\end{eqnarray}
where $K$ is some constant. Hence, we see that $\{\bar{\omega}^i_{m(l)}\}_{l\in\N}$ is a bounded sequence of $H^1_0(\Omega)^3$. Therefore, with reasoning similar to the proof of Theorem \ref{weak convergence},  we find $\bar{\omega}^i\in H^1_0(\Omega)^3$ such that $ \bar{\omega}^i_{m(l)}\wto \bar{\omega}^i$ in $H^1_0(\Omega)^3$ as $l\to\infty$, as well as $ \bar{\omega}^i_{m(l)}\wto \bar{\omega}^i$,  $\partial_{x_j}\bar{\omega}^i_{m(l)}\wto \partial_{x_j}\bar{\omega}^i$ in $L^2(\Omega)^3$ as $l\to\infty$ (up to subsequence). On the other hand, due to the Rellich-Kondrachov theorem, taking a subsequence if necessary, we see that $ \bar{\omega}^i_{m(l)}\to \bar{\omega}^i$ strongly in $L^2(\Omega)^3$ as $l\to\infty$. By \eqref{616161}, we have 
\begin{eqnarray}\label{6015}
\mbox{ $\omega^i_{m(l)}\to \bar{\omega}^i$ strongly in $L^2(\Omega)^3$ as $l\to\infty$.}
 \end{eqnarray}
Furthermore, it follows from \eqref{6-key} and \eqref{412} that for each $\phi\in C^\infty_0(\Omega)$, 
\begin{eqnarray*}
|(\D\cdot \tilde{w}^i_{m(l)},\phi)_{\Omega_{h_{\alpha^i(m(l))}}}|&\le&O( \sqrt{\tau_{\alpha^i(m(l))}})\to0\mbox{ as $l\to\infty$.} 
\end{eqnarray*}
Therefore, we obtain
\begin{eqnarray*}
&&(\D\cdot \tilde{w}^i_{m(l)},\phi)_{\Omega_{h_{\alpha^i(m(l))}}}
=-(\tilde{w}^i_{m(l)},\D\phi)_{\Omega_{h_{\alpha^i(m(l))}}}
=-(\omega^i_{m(l)},\nabla\phi)_{L^2(\Omega)^3} +O(h_{\alpha^i(m(l))})\\
&&\quad \to -(\bar{\omega}^i,\nabla\phi)_{L^2(\Omega)^3} =(\nabla\cdot \bar{\omega}^i,\phi)_{L^2(\Omega)^3}=0
\mbox{ as $l\to\infty$}, 
\end{eqnarray*}
to conclude that $\bar{\omega}:=\bar{\omega}^1-\bar{\omega}^2\in H^1_{0,\sigma}(\Omega)$.

It follows from \eqref{6161} that we have
\begin{eqnarray}\label{key1}
&&2\ge\norm \omega^1_{m(l)}- \omega^2_{m(l)} \norm_{L^2(\Omega)^3}> \eta_0 +l \nnorm \tilde{w}^1_{m(l)} -\tilde{w}^2_{m(l)}\nnorm_{op}\ge\eta_0>0\mbox{\quad for all $l\in\N$},\\\label{key2}
&& \nnorm \tilde{w}^1_{m(l)} -\tilde{w}^2_{m(l)}\nnorm_{op}\to 0\mbox{\quad as $l\to\infty$.}
\end{eqnarray} 
 For each $\phi\in C^4_{0,\sigma}(\Omega)$ with  $\norm \phi \norm_{W^{4,\infty}(\Omega)^3} =1$, we have for sufficiently large $l$, 
\begin{eqnarray*}
\nnorm \tilde{w}^1_{m(l)}\!\!\! &-&\!\!\!\tilde{w}^2_{m(l)}\nnorm_{op}\ge
 \Big|(\tilde{w}_{m(l)}^1,  Q_{h_{\alpha^1(m(l))}}\phi)_{\Omega_{h_{\alpha^1(m(l))}}}-(\tilde{w}_{m(l)}^2,  Q_{h_{\alpha^2(m(l))}}\phi)_{\Omega_{h_{\alpha^2(m(l))}}}\Big|\\
 &=&
 \Big|(\tilde{w}_{m(l)}^1,  \phi )_{\Omega_{h_{\alpha^1(m(l))}}}
 -(\tilde{w}_{m(l)}^2,  \phi )_{\Omega_{h_{\alpha^2(m(l))}}}\\
&& +  (\tilde{w}_{m(l)}^1,  Q_{h_{\alpha^1(m(l))}}\phi -\phi )_{\Omega_{h_{\alpha^1(m(l))}}} -(\tilde{w}_{m(l)}^2,  Q_{h_{\alpha^2(m(l))}}\phi-\phi )_{\Omega_{h_{\alpha^2(m(l))}}}  \Big|\\
&\to&\Big|    (\bar{\omega}^1,\phi)_{L^2(\Omega)^3}-  (\bar{\omega}^2,\phi)_{L^2(\Omega)^3} \Big|\mbox{ as $l\to\infty$ \,\,\,due to \eqref{6015}}.
\end{eqnarray*}
Hence, with \eqref{6015},  (\ref{key1}) and (\ref{key2}),  we obtain      
$$ 0<\eta_0\le\norm \bar{\omega}\norm_{L^2(\Omega)^3},\quad(\bar{\omega},\phi)_{L^2(\Omega)^3}=0\mbox{ for all $\phi\in  C^4_{0,\sigma}(\Omega)$.}$$
Therefore, $\bar{\omega}\neq0$. However,  since $\bar{\omega}\in H^1_{0,\sigma}(\Omega)$, our taking $\omega_l\in C^\infty_{0,\sigma}(\Omega)$ that approximate $\bar{\omega}$ in the $H^1$-norm as $l\to\infty$ yields 
\begin{eqnarray*}
(\bar{\omega},\bar{\omega})_{L^2(\Omega)^3}=(\bar{\omega},\omega_l)_{L^2(\Omega)^3}+(\bar{\omega},\bar{\omega}-\omega_l)_{L^2(\Omega)^3}=(\bar{\omega},\bar{\omega}-\omega_l)_{L^2(\Omega)^3}\to0\mbox{\quad as $l\to\infty$}.
\end{eqnarray*}
This is a contradiction and there exists $A_\eta=A_\eta(t)>0$ for each $t\in[0,T]$ as claimed. 

We prove that  there exists  $A_\eta>0$ independent of the choice of  $t\in[0,T]$.  Fix any $\eta>0$. Let $A^\ast_\eta(t)$ be the infimum of the set $\{A_\eta\,|\, \mbox{\eqref{key} holds} \}$ for each fixed $t$. We will prove that $A^\ast_\eta(t)$ is bounded on $[0,T]$. 
Suppose that  $A^\ast_\eta(t)$ is not bounded. Then, there exists a sequence $\{t_l\}\subset[0,T]$ for which $A_\eta^\ast(t_l)\nearrow\infty$ as $l\to\infty$.  Set $B_l:=A^\ast_\eta(t_l)/2$. For each $l\in \N$, there exists $m(l)$ for which we have   
\begin{eqnarray*}
\norm w^1_{m(l)}(t_l)- w^2_{m(l)}(t_l) \norm_{L^2(\Omega)^3}> \eta (\nnorm  w^1_{m(l)}(t_l)\nnorm+\nnorm  w^2_{m(l)}(t_l) \nnorm)+B_l \nnorm  w^1_{m(l)}(t_l)- w^2_{m(l)}(t_l) \nnorm_{op}.
\end{eqnarray*}
If $\{m(l)\}_{l\in\N}$ is unbounded, noting that $B_l\nearrow\infty$ as $l\to\infty$, we may follow the same reasoning as the above and reach a contradiction.  Suppose that $\{m(l)\}_{l\in\N}$ is bounded. Then, we have a subsequence $\{l_k\}_{k\in\N}\subset \N$ such that $m(l_k)=m_0$ for all $k\in\N$. Since $w^1_{m_0}(t)- w^2_{m_0}(t_l)$ is a step function in $t$ with a finite number of different values, we have  a constant $A_\eta>0$ such that  
$$\norm w^1_{m_0}(t)- w^2_{m_0}(t) \norm_{L^2(\Omega)^3}\le \eta (\nnorm w^1_{m_0}(t)\nnorm+\nnorm w^2_{m_0}(t) \nnorm)+A_\eta \nnorm w^1_{m_0}(t)-w^2_{m_0}(t) \nnorm_{op} $$
for all $t\in[0,T]$. However, we already obtained 
$$\norm w^1_{m_0}(t_l)- w^2_{m_0}(t_l)\norm_{L^2(\Omega)^3}> \eta (\nnorm w^1_{m_0}(t_l)\nnorm+\nnorm w^1_{m_0}(t_l) \nnorm)+B_{l_k} \nnorm w^1_{m_0}(t_l)- w^2_{m_0}(t_l) \nnorm_{op}$$
 with $B_{l_k}\to\infty$ as $k\to\infty$, and we reach a contradiction.
\end{proof}
Now, we are ready to state the result on strong convergence. 
\begin{Thm}\label{strong-convergence}
Suppose that $\delta=(h,\tau)\to0$ with the scaling condition (\ref{scaling}). Then, the sequence $\{v_\delta\}$, which is  defined in Section 5 and is weakly convergent to the  limit $v$, converges to $v$ strongly in $L^2([0,T];L^2(\Omega)^3)$.    
\end{Thm}
\begin{proof}
We re-write $\{v_\delta\}$  as $\{v_\alpha\}_{\alpha\in\N}$, where each $v_\alpha$ is defined by the difference solution $u^{n+\2}_\alpha:\Omega_{h_\alpha}\to\R^3$,  $n=0,1,2,\ldots,T_{\tau_{\alpha}}$ of the discrete Navier-Stokes equations. Suppose that $\{v_\alpha\}$ does not converge to $v$ strongly in $L^2([0,T];L^2(\Omega)^3)$ as $\alpha\to\infty$. Then, $\{v_\alpha\}$ is not a Cauchy sequence in $L^2([0,T];L^2(\Omega)^3)$, namely  there exists $\ep_0>0$ such that  for each $m\in \N$ there exist $\alpha^1(m),\alpha^2(m)\ge m$ for which  
 $0<\ep_0\le \norm v_{\alpha^1(m)}-v_{\alpha^2(m)}\norm_{L^2([0,T];L^2(\Omega)^3)}$ holds.  Lemma \ref{key-lemma} yields 
\begin{eqnarray*}
0&<&\ep_0\le \norm v_{\alpha^1(m)}-v_{\alpha^2(m)}\norm_{L^2([0,T];L^2(\Omega)^3)}\\
&\le& \eta \underline{\Big\{ \Big(\int_0^T\nnorm v_{\alpha^1(m)}(t,\cdot)\nnorm^2 dt \Big)^\2+ \Big(\int_0^T\nnorm v_{\alpha^2(m)}(t,\cdot)\nnorm^2 dt \Big)^\2  \Big\}}_{\rm(\ast)} \\
&&+A_\eta\Big(  \int_0^T\nnorm v_{\alpha^1(m)}(t,\cdot)-v_{\alpha^2(m)}(t,\cdot)\nnorm_{op}^2dt  \Big)^\2 \mbox{\quad for all $m\in\N$},
\end{eqnarray*} 
where $\eta>0$ is arbitrary and $A_\eta$ is a constant. Since
\begin{eqnarray*}
\int_0^{T}\nnorm v_{\alpha^i(m)}(t,\cdot)\nnorm^2dt &\le& 
\sum_{0\le n\le T_{\tau_{\alpha^i(m)}}} 
\Big(\norm u^{n+\2}_{\alpha^i(m)}\norm_{\Omega_{h_{\alpha^i(m)}}}^2
 +\sum_{j=1}^3\norm D^+_j u^{n+\2}_{\alpha^i(m)}\norm_{\Omega_{h_{\alpha^i(m)}}}^2 \\
 &&+ \tau_{\alpha^i(m)}\norm f_{\alpha^i(m)}^{n+1}\norm^2_{\Omega_{h_{\alpha^i(m)}}}  \Big)\tau_{\alpha^i(m)}, 
 \end{eqnarray*}
Theorem \ref{solvability} implies that the term ($\ast$) is bounded independently from $m$. Hence, $\eta\times$($\ast$) can be arbitrarily small. If we prove $\nnorm v_{\alpha^1(m)}(t,\cdot)-v_{\alpha^2(m)}(t,\cdot)\nnorm_{op}\to0$ as $m\to\infty$ for each $t\in(0,T)$, we reach a contradiction and the proof is done. 

Set 
$$l(m):=\alpha^1(m), \quad k(m):=\alpha^2(m).$$
  Fix $t\in(0,T)$. Let $n_{l(m)}$ be such that $t\in[\tau_{l(m)}n_{l(m)},\tau_{l(m)}(n_{l(m)}+1))$. For $\tilde{t}\in (t,T)$, let $\tilde{n}_{l(m)}$ be such that $\tilde{t}\in[\tau_{l(m)}(\tilde{n}_{l(m)}+1),\tau_{l(m)}(\tilde{n}_{l(m)}+2))$, where we will later choose $\tilde{t}$ close enough to $t$.  Define   
\begin{eqnarray*}
a_{l(m)}&:=&\frac{1}{\tau_{l(m)}(\tilde{n}_{l(m)}-n_{l(m)})}\sum_{n=n_{l(m)}+1}^{\tilde{n}_{l(m)}} u_{l(m)}^{n+\2} \tau_{l(m)},\\
b_{l(m)}&:=&\frac{1}{\tau_{l(m)}(\tilde{n}_{l(m)}-n_{l(m)})}\sum_{n=n_{l(m)}}^{\tilde{n}_{l(m)}-1} \tau_{l(m)}(n-\tilde{n}_{l(m)}) \frac{u_{l(m)}^{n+1+\2}-u^{n+\2}_{l(m)} }{\tau_{l(m)}} \tau_{l(m)}\\
&=&\frac{1}{\tilde{n}_{l(m)}-n_{l(m)}}\sum_{n=n_{l(m)}}^{\tilde{n}_{l(m)}-1}\Big\{ ((n+1)-\tilde{n}_{l(m)}) u_{l(m)}^{n+1+\2}- (n-\tilde{n}_{l(m)})u^{n+\2}_{l(m)} \Big\}-a_{l(m)}. 
\end{eqnarray*}
Then, we have $u_{l(m)}^{n_{l(m)}+\2}=a_{l(m)}+b_{l(m)}$ (c.f., integration by parts). We introduce $n_{k(m)}$, $\tilde{n}_{k(m)}$, $a_{k(m)}$ and  $b_{k(m)}$ in the same way with the same $t$ and $\tilde{t}$, to have $u_{k(m)}^{n_{k(m)}+\2}=a_{k(m)}+b_{k(m)}$. Setting $\phi_{l(m)}:=Q_{h_{l(m)}}\phi$, $\phi_{k(m)}:=Q_{h_{k(m)}}\phi$, we observe that 
\begin{eqnarray*}
&&\nnorm v_{l(m)}(t,\cdot)-v_{k(m)}(t,\cdot)\nnorm_{op}\\
&&= \sup_{\substack{\phi\in  C^4_{0,\sigma}(\Omega), \norm \phi \norm_{W^{4,\infty}(\Omega)^3} =1,\\ \phi=0\,\,{\rm on}
\,\,C_{2h_{l(m)}}(y)\,\,{\rm for \,\,all}\,\,y\in\partial\Omega_{h_{l(m)}},\\
\phi=0\,\,{\rm on}
\,\,C_{2h_{k(m)}}(y)\,\,{\rm for \,\,all}\,\,y\in\partial\Omega_{h_{k(m)}}
}}
\Big| (u_{l(m)}^{n_{l(m)}+\2},\phi_{l(m)})_{\Omega_{h_{l(m)}}} -  (u_{k(m)}^{n_{k(m)}+\2},\phi_{k(m)})_{\Omega_{h_{k(m)}}}\Big|,\\
&&\Big| (u_{l(m)}^{n_{l(m)}+\2},\phi_{l(m)})_{\Omega_{h_{l(m)}}} -  (u_{k(m)}^{n_{k(m)}+\2},\phi_{k(m)})_{\Omega_{h_{k(m)}}}\Big|\\
&&\le\Big|(a_{l(m)},\phi_{l(m)})_{\Omega_{h_{l(m)}}}-(a_{k(m)},\phi_{k(m)})_{\Omega_{h_{k(m)}}}\Big|+\Big| (b_{l(m)},\phi_{l(m)})_{\Omega_{h_{l(m)}}}\Big|\\
&&\quad+\Big|(b_{k(m)},\phi_{k(m)})_{\Omega_{h_{k(m)}}} \Big|.
\end{eqnarray*}
We estimate $| (b_{l(m)},\phi_{l(m)})_{\Omega_{h_{l(m)}}}|$ for each admissible test function $\phi$ with the discrete Navier-Stokes equations  \eqref{initial-1}-\eqref{fractional-3}. Hereafter, $\beta_1,\beta_2,\ldots$ are some constant independent of $m$. The discrete Helmholtz-Hodge decomposition yields $\psi^n:\Omega_{h_{l(m)}}\to\R$ such that  
\begin{eqnarray}\nonumber
u^{n}_{l(m)}(x)&=&u^{n-1+\2}_{l(m)}(x)-\D\psi^n(x),\quad x\in\Omega_{h_{l(m)}}\setminus\partial\Omega_{h_{l(m)}},\\\label{65551}
u^{n+\2}_{l(m)}(x)&=&u^{n-1+\2}_{l(m)}(x)-\tau_{l(m)} \sum_{j=1}^3\Big\{  \frac{1}{2}\Big(u^n_{l(m)j}(x-he^j)D_j^+u_{l(m)}^{n+\2}(x-h_{l(m)}e^j)\\\nonumber
&&+u^n_{l(m)j}(x)D_j^+u_{l(m)}^{n+\2}(x)\Big)-D_j^2u_{l(m)}^{n+\2} (x)\Big\}\\\nonumber
&&+\tau_{l(m)}f^{n+1}_{l(m)}(x)-\D\psi^n(x),
\mbox{\quad  $x\in\Omega_{h_{l(m)}}\setminus\partial\Omega_{h_{l(m)}}$},\\\nonumber
\sum_{x\in\Omega_{h_{l(m)}}\setminus\partial\Omega_{h_{l(m)}}} |\D\psi^n|^2&\le&\sum_{x\in\Omega_{h_{l(m)}}} | u^{n-1+\2}_{l(m)}(x)|^2.
\end{eqnarray}
By  Lemma \ref{vector-cal}, \eqref{66Taylor}, Theorem \ref{22estimate} and  Theorem \ref{solvability}, we have 
\begin{eqnarray*}
&&\frac{1}{\tau_{l(m)}}|(\D\psi^n,\phi_{l(m)} )_{h_l(m)}|= \frac{1}{\tau_{l(m)}}\Big|\sum_{x\in \Omega_{h_{l(m)}}\setminus\partial\Omega_{h_{l(m)}}}\psi^n(x)\D\cdot \phi_{l(m)}(x)h_{l(m)}^3\Big|  \\
&&\le \beta_1  \frac{h_{l(m)}^3}{\tau_{l(m)}}\Big\{\sum_{x\in\Omega_{h_{l(m)}}\setminus\partial\Omega_{h_{l(m)}}} |\psi^n(x)|^2h_{l(m)}^3\Big\}^\2\\
&&\le\beta_2  \frac{h_{l(m)}^3}{\tau_{l(m)}}\Big\{\sum_{x\in\Omega_{h_{l(m)}}} |\D\psi^n|^2h_{l(m)}^3\Big\}^\2\\
&&\le \beta_3 \frac{h_{l(m)}^3}{\tau_{l(m)}}
\le\beta_3h_{l(m)}^\alpha,
\end{eqnarray*}
where we note that $\phi\equiv0$ near $\partial\Omega$.  Hence, noting again that $\phi\equiv0$ near $\partial\Omega$ and with \eqref{65551}, we have 
\begin{eqnarray*}
&&| (b_{l(m)},\phi_{l(m)})_{\Omega_{h_{l(m)}}}|\le \sum_{n=n_{l(m)}+1}^{\tilde{n}_{l(m)}} \Big|\Big(\frac{u_{l(m)}^{n+\2}-u^{n-1+\2}_{l(m)} }{\tau_{l(m)}},\phi_{l(m)}\Big)_{\Omega_{h_{l(m)}}} \Big|\tau_{l(m)}\\
&&\le \beta_4 h_{l(m)}^\alpha+  \frac{1}{2} \underline{\sum_{n=n_{l(m)}+1}^{\tilde{n}_{l(m)}}  
\Big|\sum_{j=1}^3  \Big(u^n_{l(m)j}(\cdot-he^j)D_j^+u_{l(m)}^{n+\2}(\cdot-h_{l(m)}e^j)}\\
 &&\underline{+u^n_{l(m)j}(\cdot)D_j^+u_{l(m)}^{n+\2}(\cdot),\phi_{l(m)}(\cdot)\Big)_{\Omega_{h_{l(m)}}}   \Big|\tau_{l(m)}  }_{R_1} \\
&&+  \underline{ \sum_{n=n_{l(m)}+1}^{\tilde{n}_{l(m)}}\Big|  \sum_{j=1}^3( D_j^2u_{l(m)}^{n+\2} ,\phi_{l(m)})_{\Omega_{h_{l(m)}}}   \Big|\tau_{l(m)}  }_{R_2}
+\underline{\sum_{n=n_{l(m)}+1}^{\tilde{n}_{l(m)}}  \Big|  (f_{l(m)}^{n+1},\phi_{l(m)})_{\Omega_{h_{l(m)}}}  \Big|\tau_{l(m)}  }_{R_3}.
\end{eqnarray*}
We estimate the terms $R_1,R_2,R_3$: Noting that $\norm\phi\norm_{W^{4,\infty}(\Omega)^3}= 1$, we obtain with \eqref{412} and \eqref{413},  
\begin{eqnarray*}
R_1&\le&\beta_5\sum_{n=n_{l(m)}+1}^{\tilde{n}_{l(m)}}  \sum_{j=1}^3\norm u^n_{l(m)} \norm_{h_{l(m)}} \norm D_j^+u^{n+\2}_{l(m)}\norm_{h_{l(m)}}\tau_{l(m)}  \\
&\le& \beta_6\sum_{n=n_{l(m)}+1}^{\tilde{n}_{l(m)}} \sum_{j=1}^3 \norm D_j^+u^{n+\2}_{l(m)}\norm_{h_{l(m)}}\tau_{l(m)} \le \beta_7 \sqrt{\tilde{t}-t},\\
R_2&=&   \sum_{n=n_{l(m)}+1}^{\tilde{n}_{l(m)}}\Big|  \sum_{j=1}^3( D_j^+u_{l(m)}^{n+\2} ,D_j^+\phi_{l(m)})_{h_{m(l)}}    \Big|\tau_{l(m)} \\
&\le& \beta_8\sum_{n=n_{l(m)}+1}^{\tilde{n}_{l(m)}} \sum_{j=1}^3 \norm D_j^+u^{n+\2}_{l(m)}\norm_{h_{l(m)}}\tau_{l(m)}\le\beta_9\sqrt{\tilde{t}-t}\\
R_3&\le& \beta_{10}\sum_{n=n_{l(m)}+1}^{\tilde{n}_{l(m)}} \norm f_{l(m)}^{n+1}\norm_{h_{l(m)}}\tau_{l(m)}  \le \beta_{11}\sqrt{\tilde{t}-t}.
\end{eqnarray*} 
Therefore, we see that for any $\ep>0$ there exists $\tilde{t}>t$ such that $| (b_{l(m)},\phi_{l(m)})_{\Omega_{h_{l(m)}}}|<\ep$ as $m\to\infty$ for all admissible $\phi$, which holds for $| (b_{k(m)},\phi_{k(m)})_{\Omega_{h_{k(m)}}}|$ as well. 
We fix such $\tilde{t}$. Since $\{v_\alpha\}_{\alpha\in\N}$ weakly converges to $v$ as $\alpha\to\infty$, we have
\begin{eqnarray*}
&&\!\!\!\!\!\!\!\!\!\!\Big|(a_{l(m)},\phi_{l(m)})_{\Omega_{h_{l(m)}}}-(a_{k(m)},\phi_{k(m)})_{\Omega_{h_{k(m)}}}\Big|=\Big|(a_{l(m)},\phi_{l(m)})_{\Omega_{h_{l(m)}}}\\
&&\quad-  \frac{1}{\tilde{t}-t}  \int_t ^{\tilde{t}} (v(s,\cdot),\phi)_{L^2(\Omega)^3} ds 
+  \frac{1}{\tilde{t}-t}  \int_t ^{\tilde{t}} (v(s,\cdot),\phi)_{L^2(\Omega)^3} ds-(a_{k(m)},\phi_{k(m)})_{\Omega_{h_{k(m)}}}\Big|\\
&&\le \beta_{12} (\tau_{l(m)}+ \tau_{k(m)}) +\Big| \frac{1}{\tilde{t}-t}  \int_t ^{\tilde{t}} (v_{l(m)}(s,\cdot)-v(s,\cdot),\phi)_{L^2(\Omega)^3} ds\Big| \\
&&\quad+ \Big| \frac{1}{\tilde{t}-t}  \int_t ^{\tilde{t}} (v_{k(m)}(s,\cdot)-v(s,\cdot),\phi)_{L^2(\Omega)^3} ds\Big|\quad\to0\mbox{\quad as $m\to\infty$,}
\end{eqnarray*}
where it is easy to check that the convergence is uniform with respect to $\phi\in C^4_{0,\sigma}(\Omega)$ with $  \norm \phi \norm_{W^{4,\infty}(\Omega)^3} =1$. Thus, we conclude that $\nnorm v_{\alpha^1(m)}(t,\cdot)-v_{\alpha^2(m)}(t,\cdot)\nnorm_{op}\to0$ as $m\to\infty$ for each $t\in(0,T)$ and we reach a contradiction.
\end{proof}

\setcounter{section}{6}
\setcounter{equation}{0}
\section{Convergence to a Leray-Hopf weak solution}

We prove that the limit $v$ of $\{u_\delta\}$ and $\{v_\delta\}$ is a Leray-Hopf weak solution of \eqref{NS}. For this purpose, we change the finite difference equations into a weak form. 

Fix an arbitrary test function $\phi\in C^\infty_0([-1,T);C^\infty_{0,\sigma}(\Omega))$. Set $\phi^n:=Q_h\phi(\tau n,\cdot):\Omega_h\to\R^3$, $n=0,1,\ldots,T_\tau$ where $Q_h$ is introduced in Section 6. Note that $\phi\equiv0$ near $\partial \Omega$ and near $t=T$. For each $n$, there exists $\psi^{n+1}:\Omega_h\to\R$ such that 
\begin{eqnarray}\label{weak-form}
(u^{n+1},\phi^{n+1})_{\Omega_h}=(u^{n+\2},\phi^{n+1})_{\Omega_h}-(\D\psi^{n+1},\phi^{n+1})_{\Omega_h}.
\end{eqnarray}
As we observed in the proof of Theorem \ref{strong-convergence}, we have 
$$|(\D\psi^{n+1},\phi^{n+1})_{\Omega_h}|=O(h^\alpha)\tau.$$
With the discrete Navier-Stokes equations, we have
\begin{eqnarray*}
&&\!\!\!\!\!\!\!\!\!\!(u^{n+\2},\phi^{n+1})_{\Omega_h}=(u^n,\phi^{n})_{\Omega_h}+(u^n,\frac{\phi^{n+1}-\phi^n}{\tau})_{\Omega_h}\tau\\
&&-\sum_{j=1}^3 \big(u^n_j(\cdot-he^j)D_j^+u^{n+\2}(\cdot-he^j)+ u^n_j(\cdot)D_j^+u^{n+\2}(\cdot),\phi^{n+1}\big)_{\Omega_h}\frac{\tau}{2}\\
&&- \sum_{j=1}^3(D^+_ju^{n+\2},D^+_j\phi^{n+1})_{\Omega_h}\tau+( f^{n+1},\phi^{n+1})_{\Omega_h}\tau.
\end{eqnarray*}
Due to the discrete divergence free constraint of $u^n$ and $\phi^{n+1}$ being $0$ near $\partial\Omega$, we have 
\begin{eqnarray*}
&&\sum_{j=1}^3 \big(u^n_j(\cdot-he^j)D_j^+u^{n+\2}(\cdot-he^j)+ u^n_j(\cdot)D_j^+u^{n+\2}(\cdot),\phi^{n+1}\big)_{\Omega_h}\\
&&=\sum_{i,j=1}^3\sum_{x\in\Omega_h} 
\Big(u^n_j(x-he^j)\frac{u^{n+\2}_i(x)-u^{n+\2}_i(x-he^j)}{h}\\
&&\qquad+u^n_j(x)\frac{u^{n+\2}_i(x+he^j)-u^{n+\2}_i(x)}{h} \Big)\phi^{n+1}_i(x)h^3\\
&&=\sum_{i,j=1}^3 \sum_{x\in\Omega_h\setminus\partial\Omega_h}
-\frac{u^n_j(x)-u^n_j(x-he^j)}{h}u^{n+\2}_i(x)\phi^{n+1}_i(x)h^3\\
&&\qquad+\sum_{i,j=1}^3 \sum_{x\in\Omega_h}
\frac{1}{h}u^n_j(x)u^{n+\2}_i(x+he^j)\phi^{n+1}_i(x)h^3\\
&&\qquad-\sum_{i,j=1}^3 \sum_{x\in\Omega_h} \frac{1}{h}u^n_j(x-he^j)u^{n+\2}_i(x-he^j)\phi^{n+1}_i(x)h^3\\
&&=\sum_{i,j=1}^3 \sum_{x\in\Omega_h}
\frac{1}{h}u^n_j(x)u^{n+\2}_i(x+he^j)\phi^{n+1}_i(x)h^3\\
&&\qquad-\sum_{i,j=1}^3 \sum_{x\in\Omega_h} \frac{1}{h}u^n_j(x)u^{n+\2}_i(x)\phi^{n+1}_i(x+he^j)h^3\\
&&=-\sum_{i,j=1}^3 \sum_{x\in\Omega_h} \frac{1}{h}(u^n_j(x)u^{n+\2}_i(x)\phi^{n+1}_i(x+he^j)  -     u^n_j(x)u^{n+\2}_i(x)\phi^{n+1}_i(x) )h^3\\
&&\qquad+\sum_{i,j=1}^3 \sum_{x\in\Omega_h}
\frac{1}{h}(u^n_j(x)u^{n+\2}_i(x+he^j)\phi^{n+1}_i(x)  -u^n_j(x)u^{n+\2}_i(x)\phi^{n+1}_i(x) )h^3\\
&&=-\sum_{i,j=1}^3 \sum_{x\in\Omega_h} u^n_j(x)u^{n+\2}_i(x)D_j^+\phi^{n+1}_i(x)h^3\\
&&\qquad+\sum_{i,j=1}^3 \sum_{x\in\Omega_h}
\frac{1}{h}(u^n_j(x-he^j)u^{n+\2}_i(x)\phi^{n+1}_i(x-he^j)  -u^n_j(x)u^{n+\2}_i(x)\phi^{n+1}_i(x) )h^3\\
&&=-\sum_{i,j=1}^3 \sum_{x\in\Omega_h} u^n_j(x)u^{n+\2}_i(x)D_j^+\phi^{n+1}_i(x)h^3\\
&&\qquad+\sum_{i,j=1}^3 \sum_{x\in\Omega_h}
\frac{1}{h}(u^n_j(x-he^j)u^{n+\2}_i(x)(\phi^{n+1}_i(x-he^j) -\phi^{n+1}_i(x))\\
&&\qquad-\sum_{i,j=1}^3 \sum_{x\in\Omega_h} \frac{ u^n_j(x)-u^n_j(x-he^j)}{h}u^{n+\2}_i(x)\phi^{n+1}_i(x) )h^3\\
&&=-\sum_{i,j=1}^3 \sum_{x\in\Omega_h} (u^n_j(x)u^{n+\2}_i(x)D_j^+\phi^{n+1}_i(x)  + u^n_j(x)u^{n+\2}_i(x+he^j)D_j^+\phi^{n+1}_i(x)  )h^3\\
&&=-\sum_{j=1}^3\Big\{\Big(u^n_j(\cdot)u^{n+\2}(\cdot),D_j^+\phi^{n+1}(\cdot) \Big)_{\Omega_h} + \Big(u^n_j(\cdot)u^{n+\2}(\cdot+he^j),D_j^+\phi^{n+1}(\cdot)  \Big)_{\Omega_h}\Big\}.
\end{eqnarray*} 
Hence, taking summation in \eqref{weak-form} with respect to $n$ and noting that $\phi^{T_\tau}=0$, we have 
\begin{eqnarray*}
&&0=(u^0,\phi^0)_{\Omega_h}+\sum_{n=0}^{T_\tau-1}  (u^n,\frac{\phi^{n+1}-\phi^n}{\tau})_{\Omega_h}\tau   \\\nonumber
&&\quad+\sum_{j=1}^3 \sum_{n=0}^{T_\tau-1}\Big\{(u^n_j(\cdot)u^{n+\2}(\cdot),D_j^+\phi^{n+1}(\cdot) )_{\Omega_h} + (u^n_j(\cdot)u^{n+\2}(\cdot+he^j),D_j^+\phi^{n+1}(\cdot)  )_{\Omega_h}\Big\}\frac{\tau}{2}\\\nonumber
&&\quad- \sum_{j=1}^3\sum_{n=0}^{T_\tau-1}(D^+_ju^{n+\2},D^+_j\phi^{n+1})_{\Omega_h}\tau +\sum_{n=0}^{T_\tau-1}( f^{n+1},\phi^{n+1})_{\Omega_h}\tau +O(h^\alpha).
\end{eqnarray*} 
Therefore, noting \eqref{412} and \eqref{413}, we obtain 
\begin{eqnarray}\label{weak-form2}
\quad0&\!\!\!=\!\!\!&
\underline{\int_\Omega u_\delta(0,x)\cdot \phi(0,x)dx }_{R_1}  + 
\underline{\int_0^T\int_\Omega u_\delta(t,x)\cdot\partial_t\phi(t,x)dxdt }_{R_2}    \\\nonumber
&& + \underline{\sum_{j=1}^3\frac{1}{2} \int_0^T\int_\Omega \Big\{ u_{\delta j}(t,x) v_{\delta}(t,x)\cdot \partial_{x_j}\phi(t,x) } \\\nonumber
&&\underline{ + u_{\delta j}(t,x)v_{\delta }(t,x+he^j)\cdot\partial_{x_j}\phi(t,x)\Big\}dxdt  }_{R_3}\\\nonumber
&&- \underline{ \sum_{j=1}^3\int_0^T\int_\Omega w^j_\delta(t,x)\cdot \partial_{x_j}\phi(t,x)dxdt }_{R_4}+
\underline{ \int_0^T\int_\Omega f_\delta(t,x)\cdot\phi(t,x)dxdt }_{R_5}\\\nonumber
&&+O(h^\alpha)+O(h),
\end{eqnarray} 
where $f_\delta:[0,T]\times\Omega\to\R^3$ is the step function defined by $f^{n+1}:\Omega\to\R^3$, $n=0,1,\ldots,T_\tau-1$ as
\begin{eqnarray*}
 f_\delta(t,x):=\left\{\!\!\!\!\!\!
\begin{array}{lll}
&f^{n+1}(y)\mbox{\quad for  $(t,x)\in (\tau n,\tau n+\tau]\times C^+_h(y)$, $y\in\Omega_h$},
\medskip\\
&0\mbox{\quad\,\,\,\,\,\,\,\,\,\,\,\,\,\,\, otherwise}.
\end{array}
\right. \\
\end{eqnarray*}
\begin{Thm}
The limit $v$ of $\{u_\delta\}$ and $\{v_\delta\}$ derived with \eqref{scaling} is a Leray-Hopf weak solution of (\ref{NS}).
\end{Thm}
\begin{proof}
We already proved that $v\in L^2([0,T];H^1_{0,\sigma}(\Omega))$ in Section 5. Since  $v_\delta$ converges to $v$ strongly in $L^2([0,T];L^2(\Omega)^3)$ as $\delta\to0$ and $v_\delta$ belongs to $L^\infty([0,T];L^2(\Omega)^3)$ with $ \norm v_\delta(t,\cdot)\norm_{L^2(\Omega)^3}$ bounded independently from $t\in[0,T]$ and $\delta$, 
we see that $v\in  L^\infty([0,T];L^2_{\sigma}(\Omega))$ by taking an a.e. pointwise convergent subsequence of $\{v_\delta\}$.

We show that \eqref{weak-form2} yields \eqref{weak-form-NS} as the limit of $\delta\to0$. It follows from Theorem \ref{weak convergence} that 
$$R_2\to \int_0^T\int_\Omega v\cdot \partial_t\phi \,dxdt,\quad R_4\to \sum_{j=1}^3\int_0^T\int_\Omega \partial_{x_j}v\cdot\partial_{x_j}\phi\, dxdt\mbox{\quad as $\delta\to0$}.$$
With  \eqref{412},  \eqref{413},  Theorem \ref{weak convergence} and Theorem \ref{strong-convergence},  we have 
\begin{eqnarray*}
&&\!\!\!\!\!\!\!\Big|\int_0^T\int_\Omega u_{\delta j}(t,x) v_{\delta}(t,x+he^j)\cdot \partial_{x_j}\phi(t,x) -v_{ j}(t,x) v(t,x)\cdot \partial_{x_j}\phi(t,x)  dxdt\Big|\\
&&\!\!\!\!\!\!\!=\Big|\int_0^T\int_\Omega \Big\{    u_{\delta j}(t,x)( v(t,x) \cdot\partial_{x_j}\phi(t,x)) -v_j(t,x)( v(t,x) \cdot\partial_{x_j}\phi(t,x))\\
&&\!\!\!\!\!\!\!\quad+  u_{\delta j}(t,x)\big(v_\delta(t,x+he^j)-v(t,x+he^j)\big)\cdot \partial_{x_j}\phi(t,x) \\
&&\!\!\!\!\!\!\!\quad+    u_{\delta j}(t,x)\big(v(t,x+he^j)-v(t,x)\big)\cdot \partial_{x_j}\phi(t,x)\Big\}  dxdt\Big|\\
&&\!\!\!\!\!\!\!   \le    
\Big|\int_0^T\int_\Omega \Big\{    u_{\delta j}(t,x)( v(t,x) \cdot\partial_{x_j}\phi(t,x)) -v_j(t,x)( v(t,x) \cdot\partial_{x_j}\phi(t,x))dxdt\Big|\\
&&\!\!\!\!\!\!\!\quad+ \norm u_{\delta j}\partial_{x_j}\phi\norm_{L^2([0,T];L^2(\Omega))}\norm v_{\delta }-v \norm_{L^2([0,T];L^2(\Omega)^3)}\\
&&\!\!\!\!\!\!\!\quad+  \norm u_{\delta j}\partial_{x_j}\phi\norm_{L^2([0,T];L^2(\Omega))}\norm v(\cdot,\cdot-he^j)-v(\cdot,\cdot) \norm_{L^2([0,T];L^2(\Omega)^3)}\\
&&\!\!\!\!\!\!\!\to0\mbox{\qquad as $\delta\to0$}.
\end{eqnarray*}
Taking care of the first term in $R_3$ in the same way,  we obtain
\begin{eqnarray*}
&&R_3\to \sum_{j=1}^3 \int_0^T\int_\Omega    v_{ j}(t,x)v(t,x) \cdot  \partial_{x_j}\phi(t,x)    dxdt\\
&&\qquad\qquad =-\sum_{j=1}^3\int_0^T\int_\Omega v_j(t,x)\partial_{x_j}v(t,x)\cdot\phi(t,x)dxdt   \mbox{ \quad as $\delta\to0$.}
\end{eqnarray*} 
To examine $R_1$, we take a sequence $\{v^{0l}\}_{l\in\N}\subset C^{\infty}_{0,\sigma}(\Omega)$ that converges to $v^0$ in $L^2(\Omega)^3$ as $l\to\infty$. For each $\delta=(h,\tau)$, define $u^{0l},\tilde{u}^{0l}:\Omega_h\to\R^3$ and $u^{0l}_h,\tilde{u}^{0l}_h:\Omega\to\R^3$ as  
\begin{eqnarray*}
&&\tilde{u}^{0l}(y):=\frac{1}{h^3}\int_{C_h(y)}v^{0l}(z)dz\mbox{ for $y\in\Omega_h$,\quad}\\
&& \tilde{u}^{0l}_h(x):=\left\{\!\!\!\!\!\!
\begin{array}{lll}
&\tilde{u}^{0l}(y)\mbox{\quad for  $x\in {C^+_h(y)}$, $y\in \Omega_h$},
\medskip\\
&0\mbox{\quad\,\,\,\,\,\,\,\,\,\, otherwise},
\end{array}
\right. \\
&&u^{0l}:=P_h\tilde{u}^{0l},\\
 &&u^{0l}_h(x):=\left\{\!\!\!\!\!\!
\begin{array}{lll}
&u^{0l}(y)\mbox{\quad for  $x\in  C_h^+(y)$, $y\in\Omega_h$},
\medskip\\
&0\mbox{\quad\,\,\,\,\,\,\,\,\,\, otherwise}.
\end{array}
\right. \\\end{eqnarray*}
Then, we have 
\begin{eqnarray*}
&&\Big|  \int_\Omega u_\delta(0,x)\cdot\phi(0,x) -v^0(x)\cdot\phi(0,x)dx  \Big|\le   \big\{ \norm u_\delta(0,x)-u^{0l}_h \norm_{L^2(\Omega)^3}\\
&&\quad +\norm u^{0l}_h-\tilde{u}^{0l}_h   \norm_{L^2(\Omega)^3}+\norm \tilde{u}^{0l}_h-v^{0l}  \norm_{L^2(\Omega)^3}+\norm v^{0l}-v^0  \norm_{L^2(\Omega)^3}\big\} \norm\phi(0,\cdot)\norm_{L^2(\Omega)^3}\\
&&=\big\{ \norm P_h\tilde{u}^0-P_h\tilde{u}^{0l} \norm_{\Omega_h}+\norm P_h\tilde{u}^{0l}-\tilde{u}^{0l}   \norm_{\Omega_h}+\norm \tilde{u}^{0l}_h-v^{0l}  \norm_{L^2(\Omega)^3}\\
&&\quad+\norm v^{0l}-v^0  \norm_{L^2(\Omega)^3}\big\} \norm\phi(0,\cdot)\norm_{L^2(\Omega)^3}.
\end{eqnarray*}
With Theorem \ref{22estimate}, we have
\begin{eqnarray*}
&&\norm P_h\tilde{u}^0-P_h\tilde{u}^{0l}\norm_{\Omega_h}^2\le \norm \tilde{u}^0-\tilde{u}^{0l}\norm_{\Omega_h}^2=\sum_{y\in\Omega_h}h^{-3}\Big|     \int_{C_h(y)}v^0(z)dz-\int_{C_h(y)}v^{0l}(z)dz   \Big|^2\\
&&\le \sum_{y\in\Omega_h}h^{-3}\Big\{  \sqrt{ \int_{C_h(y)}|v^0(z)-v^{0l}(z)|^2dz}\sqrt{ \int_{C_h(y)}1dz}   \Big\}^2\le\norm v^0-v^{0l}\norm_{L^2(\Omega)^3}^2.
\end{eqnarray*}
For any $\ep>0$, we fix $l$ in such a way that $\norm v^0-v^{0l}\norm_{L^2(\Omega)^3}<\ep$. Since $v^{0l}$ belongs to $C^{\infty}_{0,\sigma}(\Omega)$, we have $\norm \tilde{u}^{0l}_h-v^{0l}\norm_{L^2(\Omega)^3}\to0$ as $h\to0$; Noting that $\tilde{u}^{0l}\equiv0$ on some neighborhood of $\partial\Omega$,   we see  that 
\eqref{242424242} is applicable  for sufficiently small $h$ to obtain 
\begin{eqnarray*}
\norm P_h\tilde{u}^{0l}-\tilde{u}^{0l}\norm_{\Omega_h}^2 &\le& A\sum_{y\in\Omega_h}| \D\cdot \tilde{u}^{0l}(y)|^2\\
&\le& A\sum_{y\in\Omega_h}\Big|   \sum_{j=1}^3 h^{-3}\int_{C_h(y)}\frac{v^{0l}_j(z)-v^{0l}_j(z-he^j)}{h}dz \Big|^2 h^3\\
&=&A\sum_{y\in\Omega_h}\Big|   h^{-3}\int_{C_h(y)}  \nabla\cdot v^{0l}(z)+O(h) dz \Big|^2h^3\\
&=&O(h^2)\to0\mbox{\quad as $h\to0$}.
\end{eqnarray*}
Therefore, we conclude that 
$$R_1\to\int_\Omega v^0(x)\cdot\phi(0,x)dx\mbox{\quad as $\delta\to0$.}$$
Take $\{f^l\}_{l\in\N}\subset C^\infty_0([0,T]\times\Omega)^3$ that converges to  $f$ in $L^2([0,T];L^2(\Omega)^3)$ as $l\to\infty$.  For each $\delta=(h,\tau)$, define  $f^{l,n+1}:\Omega_h\to\R^3$, $n=0,1,\ldots,T_\tau-1$  and $f^{n+1}_\delta:[0,T]\times\Omega\to\R^3$ as  
\begin{eqnarray*}
f^{l,n+1}(y)&:=&\frac{1}{\tau h^3}\int_{\tau n}^{\tau(n+1)}\int_{C_h(y)} f^l(s,z)dzds\mbox{ for $y\in\Omega_h$,}
\\
 f^{l}_\delta(t,x)&:=&\left\{\!\!\!\!\!\!
\begin{array}{lll}
&f^{l,n+1}(y)\mbox{\quad for  $(t,x)\in (\tau n,\tau n+\tau]\times C_h^+(y)$, $y\in\Omega_h$},
\medskip\\
&0\mbox{\quad\,\,\,\,\,\,\,\,\,\,\,\,\,\,\, \,\,\,\,otherwise}.
\end{array}
\right. \\
\end{eqnarray*}
Then, we see that 
\begin{eqnarray*}
&&\Big|  \int_0^T\int_\Omega f_\delta(t,x)\cdot\phi(t,x)-f(t,x)\cdot\phi(t,x) dxdt   \Big|\\
&&\qquad\qquad \le \norm \phi \norm_{L^2([0,T];L^2(\Omega)^3)}  \norm f_\delta-f \norm_{L^2([0,T];L^2(\Omega)^3)},\\
&&\norm f_\delta-f \norm_{L^2([0,T];L^2(\Omega)^3)}\le \norm f_\delta-f^l_\delta \norm_{L^2([0,T];L^2(\Omega)^3)} +\norm f^l_\delta-f^l \norm_{L^2([0,T];L^2(\Omega)^3)}\\
&&\qquad\qquad +\norm f^l-f \norm_{L^2([0,T];L^2(\Omega)^3)},\\
&&\norm f_\delta-f^l_\delta \norm_{L^2([0,T];L^2(\Omega)^3)}^2 \le \sum_{n=0}^{T_\tau-1}\sum_{y\in\Omega_h}\Big| \frac{1}{\tau h^3} \int_{\tau n}^{\tau(n+1)}\int_{C_h(y)}f(s,z)-f^l(s,z)dxds  \Big|^2 \tau h^3\\
&&\qquad\qquad\le\norm f^l-f\norm_{L^2([0,T];L^2(\Omega)^3)}^2. 
\end{eqnarray*}
 For any $\ep>0$, we fix $l$ in such a way that $\norm f^l-f\norm_{L^2([0,T];L^2(\Omega)^3)}<\ep$. Since $f^l$ is $C^\infty$, we have $\norm f^l_\delta-f^l \norm_{L^2([0,T];L^2(\Omega)^3)}\to0$ as $\delta\to0$. Therefore, we conclude that 
 $$R_5\to\int_0^T\int_\Omega f(t,x)\cdot\phi(t,x)dxdt\mbox{\quad as $\delta\to0$.}$$ 
\end{proof}
\appendix
\def\thesection{Appendix}
\section{}

\noindent {\bf(1) Lipschitz interpolation of step function with $0$-boundary condition.}

 {\it  For a function $u:\Omega_h\to\R$ with $u|_{\partial\Omega_h}=0$ and the step function $v$ defined as 
\begin{eqnarray*}
v(x)&:=&\left\{
\begin{array}{lll}
&u(y)\mbox{\quad for $x\in {C_{h}^+(y)}$, $y\in \Omega_{h}$},
\medskip\\
&0\mbox{\quad \,\,\,\,\,\,\,\,\mbox{otherwise}}, 
\end{array}
\right. \\
\end{eqnarray*}
there exists  a Lipschitz continuous function $w:\Omega \to \R$ with a compact support  such that 
\begin{eqnarray*}
&&\norm w-v\norm_{L^2(\Omega)}\le K h\norm \D u\norm_{\Omega_h},\\
&&\norm \partial_{x_i}w(x)\norm_{L^2(\Omega)}\le \tilde{K}\norm \D u\norm_{\Omega_h}, \mbox{ $i=1,2,3$},
\end{eqnarray*}
where $K$ and $\tilde{K}$ are constants independent of $u$ and $h$. }
\begin{proof} 
 Let $y=(y_1,y_2,y_3)\in\Omega_h$ and let $u=0$ outside $\Omega_h$. Define the functions
\begin{eqnarray*}
&&f_1(x_1):[y_1,y_1+h]\to\R,\quad \\
&&f_1(x_1):=u(y)+\frac{u(y+he^1)-u(y)}{h}(x_1-y_1);\\
&&f_2(x_1):[y_1,y_1+h]\to\R,\quad \\
&&f_2(x_1):=u(y+he^2)+\frac{u(y+he^2+he_1)-u(y+he^2)}{h}(x_1-y_1);\\
&&f_3(x_1,x_2):[y_1,y_1+h]\times[y_2,y_2+h]\to\R, \quad \\
&&f_3(x_1,x_2):= f_1(x_1)+\frac{f_2(x_1)-f_1(x_1)}{h}(x_2-y_2);\\
&&g_1(x_1):[y_1,y_1+h]\to\R,\quad \\
&&g_1(x_1):=u(y+he^3)+\frac{u(y+he^3+he^1)-u(y+he^3)}{h}(x_1-y_1);\\
&&g_2(x_1):[y_1,y_1+h]\to\R,\quad \\
&&g_2(x_1):=u(y+he^3+he^2)+\frac{u(y+he^3+he^2+he_1)-u(y+he^3+he^2)}{h}(x_1-y_1);\\
&&g_3(x_1,x_2):[y_1,y_1+h]\times[y_2,y_2+h]\to\R, \quad \\
&&g_3(x_1,x_2):= g_1(x_1)+\frac{g_2(x_1)-g_1(x_1)}{h}(x_2-y_2);\\ 
&&w(x_1,x_2,x_3):C_h^+(y)\to\R,\\
&&w(x_1,x_2,x_3):=f_3(x_1,x_2)+\frac{g_3(x_1,x_2)-f_3(x_1,x_2)}{h}(x_3-y_3)
\end{eqnarray*}
Then, we see that 
\begin{eqnarray*}
&&\!\!\!\!\!w(x_1,x_2,x_3)= u(y)+D_1^+u(y)(x_1-y_1)+ D_2^+u(y)(x_2-y_2)+D_3^+u(y)(x_3-y_3)\\
&&\qquad+\{ D_1^+u(y+he^2) -  D_1^+u(y) \}\frac{(x_1-y_1)(x_2-y_2)}{h}\\
&&\qquad +\{ D_1^+u(y+he^3)  -D_1^+u(y)\} \frac{(x_1-y_1)(x_3-y_3)}{h}\\
&&\qquad + \{D_2^+u(y+he^3)-D_2^+u(y)\}\frac{(x_2-y_2)(x_3-y_3)}{h}\\
&&\qquad +\{ D_1^+u(y+he^2+he^3)-D_1^+u(y+he^3) \\
&&\qquad -D_1^+u(y+he^2)+D_1^+u(y)\}\frac{(x_1-y_1)(x_2-y_2)(x_3-y_3)}{h^2}. 
%
\end{eqnarray*}
It is clear that $w$ can be Lipschitz continuously connected with each other, yielding $w:\Omega\to\R$ with a compact support.   Moreover, it holds that  
\begin{eqnarray*}
\int_{C^+_h(y)}|w(x)-v(x)|^2dx &\le& (|\D u(y)|^2 +|\D u(y+he^2)|^2+|\D u(y+he^3)|^2\\
&&+|\D u(y+he^2+he^3)|^2 )O(h^2) h^3,\\
\int_{C^+_h(y)}|\partial_{x_i}w(x)|^2dx &\le& 9(|D_i^+u(y)|^2+|\D u(y)|^2 +|\D u(y+he^2)|^2+|\D u(y+he^3)|^2\\
&&+|\D u(y+he^2+he^3)|^2 ) h^3.
\end{eqnarray*}
This concludes the assertion. 
\end{proof}
\medskip
\medskip

%
 %

\noindent {\bf (2) Discrete Poincar\'e type inequality.}

{\it There exists a constant $A>0$ depending only on $\Omega$ for which each function $\phi:\Omega_h\to\R$ with $\phi|_{\partial\Omega_h}=0$ satisfies 
\begin{eqnarray*}
\sum_{x\in\Omega_h} |\phi(x)|^2 \le A \sum_{x\in\Omega_h\setminus\partial\Omega_h} |D_1^+ \phi(x)|^2\le A\sum_{x\in\Omega_h\setminus\partial\Omega_h} |\D \phi(x)|^2.
\end{eqnarray*}}
\begin{proof} For each $x=(x_1,x_2,x_3)\in\Omega_h\setminus\partial\Omega_h$, there exists $J\in\N$ such that $x+hje^1\in\Omega_h\setminus\partial\Omega_h$ for $j=1,\ldots,J-1$ and $x+hJe^1\in\partial\Omega_h$. Since $\Omega$ is bounded, we have $J^\ast$ and $a>0$ such that $J\le J^\ast$ and $hJ^\ast\le a$ for all $x$, i.e., $a$ is the diameter of $\Omega$ in the $x_1$-direction. Observe that
\begin{eqnarray}\label{app2}\nonumber
(\ast)\qquad\qquad\qquad\qquad\phi(x_1,x_2,x_3)&=&-\sum_{j=0}^{J-1} D_1^+\phi(x_1+hj,x_2,x_3)h,\\\nonumber
|\phi(x_1,x_2,x_3)|^2&\le& \Big\{ \sum_{j=0}^{J-1} |D^+_1\phi(x_1+hj,x_2,x_3)| h \Big\}^2\\\nonumber
&\le&\Big\{ \Big( \sum_{j=0}^{J-1} |D^+_1\phi(x_1+hj,x_2,x_3)|^2 h\Big)^\2\Big( \sum_{j=0}^{J-1}h\Big)^\2 \Big\}^2\\\nonumber
&\le& a  \sum_{j=0}^{J-1} |D_1^+\phi(x_1+hj,x_2,x_3)|^2h,\\\nonumber
&\le& a \sum_{y_1\in\{y_1\,|\,(y_1,x_2,x_3)\in\Omega_h\setminus\partial\Omega_h\}} |D_1^+\phi(y_1,x_2,x_3)|^2h,\\\nonumber
\sum_{y_1\in\{y_1\,|\,(y_1,x_2,x_3)\in\Omega_h\setminus\partial\Omega_h\}}
\!\!\!\!\!\!\!\!\!\!\!|\phi(y_1,x_2,x_3)|^2&\le& \Big(a  \sum_{y_1\in\{y_1\,|\,(y_1,x_2,x_3)\in\Omega_h\setminus\partial\Omega_h\}}\!\!\!\!\!\!\!\!\!\!\! |D_1^+\phi(y_1,x_2,x_3)|^2h\Big)J^\ast \\\nonumber
&\le& a^2  \sum_{y_1\in\{y_1\,|\,(y_1,x_2,x_3)\in\Omega_h\setminus\partial\Omega_h\}}\!\!\!\!\!\!\!\!\!\!\! |D_1^+\phi(y_1,x_2,x_3)|^2
\end{eqnarray} 
Therefore, taking summation of the last inequality with respect to $(x_2,x_3)$ such that there exists at least one $y_1$ satisfying $(y_1,x_2,x_3)\in\Omega_h\setminus\partial\Omega_h$, we obtain our conclusion. \end{proof}
\end{document}